\documentclass[11pt,reqno]{amsart}

\usepackage{amsmath}
\usepackage{amsthm}
\usepackage{graphicx}
\usepackage{amsfonts}
\usepackage{amssymb}
\usepackage{hyperref}
\usepackage{bm}
\usepackage{bbm}
\usepackage{xcolor}
\usepackage[margin=1.25in]{geometry}

\usepackage{enumerate}
\usepackage{enumitem}
\usepackage{mathtools}
\numberwithin{equation}{section}

\numberwithin{equation}{section}

\newtheorem{theorem}{Theorem}[section]
\newtheorem{lemma}[theorem]{Lemma}
\newtheorem{proposition}[theorem]{Proposition}
\newtheorem{corollary}[theorem]{Corollary}

\theoremstyle{definition}
\newtheorem{example}[theorem]{Example}
\newtheorem{definition}[theorem]{Definition}
\newtheorem{remark}[theorem]{Remark}

\newtheorem*{assumption}{Standing assumptions}

\def\E{{\mathbb E}}

\def\R{{\mathbb R}}
\def\N{{\mathbb N}}

\def\PP{{\mathbb P}}
\def\FF{{\mathbb F}}
\def\P{{\mathcal P}}
\def\RC{{\mathcal R}}

\def\V{{\mathcal V}}
\def\M{{\mathcal M}}

\def\L{{\mathcal L}}

\def\A{{\mathcal A}}

\def\F{{\mathcal F}}
\def\tr{{\mathrm{Tr}}}

\def\Var{{\mathrm{Var}}}

\def\Punif{\P_{\mathrm{Unif}}}
\def\Aunif{\A_U}
\def\Mui{M^{n,\bm{u}}_i} 
\def\Mni{M^{n,i}} 
\def\C{{\mathcal C}}

\title[A label-state formulation of stochastic graphon games]{A label-state formulation of stochastic graphon games and approximate equilibria on large networks}
\author{Daniel Lacker}
\author{Agathe Soret}
\email{daniel.lacker@columbia.edu}
\email{acs2298@columbia.edu}
\address{Department of Industrial Engineering \& Operations Research, Columbia University}
\thanks{D.L. is partially supported by the AFOSR Grant FA9550-19-1-0291 and the NSF CAREER award DMS-2045328.}

\begin{document} 

\begin{abstract}
This paper studies stochastic games on large graphs and their graphon limits. We propose a new formulation of graphon games based on a single typical player's label-state distribution. In contrast, other recently proposed models of graphon games work directly with a continuum of players, which involves serious measure-theoretic technicalities. In fact, by viewing the label as a component of the state process, we show in our formulation that graphon games are a special case of mean field games, albeit with certain inevitable degeneracies and discontinuities that make most existing results on mean field games inapplicable. Nonetheless, we prove existence of Markovian graphon equilibria under fairly general assumptions, as well as uniqueness under a monotonicity condition. Most imporantly, we show how our notion of graphon equilibrium can be used to construct approximate equilibria for large finite games set on any (weighted, directed) graph which converges in cut norm. The lack of players' exchangeability necessitates a careful definition of approximate equilibrium, allowing heterogeneity among the players' approximation errors, and we show how various regularity properties of the model inputs and underlying graphon lead naturally to different strengths of approximation. 
\end{abstract}

\maketitle

\section{Introduction}

This paper is about network-based generalizations of the now-standard \emph{mean field game} (MFG) framework.
The latter was introduced in \cite{huang2006large,lasry-lions} to describe the large-$n$ equilibrium behavior of certain $n$-player stochastic games. Remarkably, the limiting models in MFG theory are typically quite tractable, and for this reason MFG theory developed a rich mathematical theory and a broad range of applications. However, the MFG framework is fundamentally limited to  games in which \emph{players interact symmetrically}. On the one hand, MFG models can already incorporate heterogeneity in \emph{individual} characteristics (and are often known to economists as \emph{heterogeneous agent models}), in the sense that players may face independent sources of randomness and perhaps their own \emph{type} parameters. On the other hand, MFG  theory is not well suited to modeling heterogeneity in the \emph{interactions} between players, where distinct pairs of players have different interaction strengths. Heterogeneous interactions are the defining feature of \emph{network games}, a well-developed framework which is widely applied in very different contexts from MFG theory; see \cite{jackson2010social}.

The range of applicability of MFG theory would increase dramatically if it could incorporate non-trivial network structures, or heterogeneous interactions, while maintaining a tractable limiting (continuum) model. This is a challenging prospect, in general, because different $n$-player networks may lead to very different limits as $n\to\infty$, especially in \emph{sparse} networks \cite{fouque-recent,lacker2021case}. A natural first step is to understand the range of network models for which the usual MFG remains the correct limit. There is reason to expect that this is the case for sufficiently \emph{dense} and \emph{approximately regular} networks. This intuition was confirmed in our recent linear-quadratic case study \cite{lacker2021case}, and by Delarue \cite{delarue2017mean} in a model set on dense Erd\H{o}s-R\'enyi graphs; our Remark \ref{re:MFcase} below gives a result of this nature as well.
Similar ideas appeared in non-game-theoretic models of interacting particle systems with interactions governed by networks, for which recent work has identified a certain universality of the mean field limit. See \cite{bhamidi2019weakly,coppini2019long,coppini2019law,delattre2016note,luccon2020quenched} for diffusive dynamics and \cite{basak2017universality} for static Ising and Potts models.

There are many network models, however, for which the usual MFG limit is not correct. 
Several different groups of researchers have recently proposed new continuum models, as alternatives to the usual MFG, based on the notion of \emph{graphons}. Graphons are natural continuum limits for large dense graphs, and we refer to Lov\'asz \cite{lovasz2012large} for an overview. Essentially, a graphon is a symmetric measurable function $W : [0,1]^2 \to [0,1]$, with $W(u,v)$  representing the edge density between vertices $u$ and $v$.
For static games based on graphons, we refer to the recent work \cite{carmona2019stochastic,parise2019graphon,parise2021analysis}, and, for dynamic games, see \cite{cui2021learning,vasal2020sequential} for discrete time and \cite{aurell2021stochastic,bayraktar2022propagation,gao2020linear,tangpi2022optimal} for continuous time.
A related but distinct notion of \emph{graphon mean field games} was developed in a recent series of papers by Caines et al.\ \cite{caines2019graphon,caines2018graphon}, in which each node in the network contains a subpopulation with its own mean field of players. 
There have been similar developments for non-game-theoretic models of interacting diffusions, with recent work \cite{bayraktar2020graphon,bet2020weakly} developing a graphon-based limit theory.

The goal of this paper is to advance the theory of graphon-based analogues of mean field games, or \emph{graphon games}. 
Most importantly, we are able to achieve a level of tractability which is comparable to traditional mean field games, in the following sense. The mean field game framework is based on a fixed point problem describing the law of the state process $X=(X_t)_{t \in [0,T]}$ of one ``typical" player, which represents a significant dimension-reduction when compared to a large $n$-player game. On the contrary, prior graphon-based models work directly with a \emph{continuum} of players, which arguably does not provide a significant simplification, and which leads to serious technical challenges discussed below. The graphon game model that we propose is a fixed point problem for the joint law of $(U,X)$, where $X$ is the state process coupled with a Unif$[0,1]$ random variable $U$, interpreted as the ``vertex" or ``label" of the player in the graphon.

In fact, we show that our notion of graphon game is equivalent to a classical MFG model in which $(U,X_t)_{t \in [0,T]}$ is treated as the state process. Whereas the MFG model can be captured by a single forward-backward partial differential equation (PDE) system on $[0,T] \times \R^d$, prior graphon-based models involve a continuum of coupled PDEs, and our model can be captured by a single forward-backward PDE on $[0,T] \times \R^{d+1}$. 
Despite this equivalence, it is only in special situations that one can directly apply prior theorems from the MFG literature; the coefficients are discontinuous unless the graphon is a continuous function, and the diffusion coefficient of $(U,X_t)_{t \in [0,T]}$ is always degenerate. Hence, although we adapt known MFG methods for our proofs (mainly \cite{lacker2015mean}), we must tailor them to the graphon setting.
Moreover, the finite games we study, which are governed by general interaction matrices which converge in cut norm, are quite different from the finite game naturally associated with the equivalent MFG, and our finite games thus require a significantly more involved construction for approximate equilibria. See Section \ref{se:PDE} for details.

Working directly with a continuum of players, driven by a continuum of independent Brownian motions $(B^u)_{u \in [0,1]}$, raises  significant technical difficulties stemming from the fact that $u \mapsto B^u(\omega)$ is not Lebesgue measurable for a.e.\ $\omega$. In a linear-quadratic setting, this issue was confronted directly in  \cite{aurell2021stochastic} via sophisticated measure-theoretic machinery, namely the notion of \emph{Fubini extensions} due to \cite{sun2006exact}. In \cite{bayraktar2020graphon} the issue was carefully avoided by arguing that the \emph{laws} $\L(X^u)$ of the state processes $(X^u)_{u \in [0,1]}$ depend measurably on $u$, and this is good enough for their purposes. Other works such as \cite{bet2020weakly} do not explicitly address this issue.
By focusing on the joint law of $(U,X)$, we avoid the technical challenges of the continuum.
Of course, a joint law of $(U,X)$ with $U$ uniform can be identified with its disintegration, i.e., the conditional law of $X$ given $U$, but this conditional law is uniquely determined only up to a.e.\ equality. Our notion of graphon game thus encodes less information than a model with a true continuum of players, as we may make statements about \emph{almost every} player but not about \emph{every} player. But this minor loss of information brings significant mathematical advantages. First, it avoids the aforementioned measure-theoretic difficulties. Second, it permits a simple topological setting, allowing us to use the weak topology on $\P([0,1] \times \R^d)$, in which compacts are far more abundant when compared to the uniform or $L^p$ topologies on spaces of functions $[0,1] \to \P(\R^d)$ employed in some prior works (e.g., \cite{caines2018graphon}).

Using our new graphon game formulation, we prove several fundamental results under fairly general assumptions on the model inputs.
First, we prove existence of an equilibrium which is \emph{Markovian} in the sense that the control is a function of $(t,U,X_t)$. 
We also show uniqueness under a graphon version of the Lasry-Lions monotonicity assumption. See Section \ref{se:existence&uniqueness} for these results.
Our new framework allows us to handle, with relative ease, far more general setups than were considered in prior work. For instance, in prior work, the interactions are \emph{pairwise}, in the sense that the effect of the other players $j \neq i$ on a player $i$ is given by a quantity of the form $n^{-1}\sum_{j=1}^n \xi_{ij} h(X^i,X^j)$, where $\xi$ is an $n \times n$ interaction matrix. More generally, we are able to treat higher order interactions depending on the empirical measure $n^{-1}\sum_{j=1}^n \xi_{ij} \delta_{X^j}$, which admits a simple continuum analogue (defined in Section \ref{se:graphon-operators}) in terms of the joint law of $(U,X)$.

Our most important results justify our new formulation by showing that any graphon game equilibrium can be used to construct approximate equilibria for the $n$-player game, when the latter involves an interaction matrix which converges to the given graphon in the cut norm (or, more generally, in the strong operator topology, though this generalization does not complicate our proofs). This is the most challenging part of our work. The precise notion of \emph{approximate equilibrium} can take various forms: the $\epsilon^n_i$ error may be different for each player $i$, and the $\epsilon^n_i$ may vanish in an averaged or uniform sense,  depending on the structural assumptions (such as continuity) imposed on the graphon. See Section \ref{subse:main_res_approx_eq} for precise statements. Prior work on graphon games, with a few exceptions, has assumed the $n$-player game to be set on a specific  exchangeable random graph ``sampled" from the graphon in the usual manner, which enjoys particularly strong convergence properties as $n\to\infty$. It is more general, and also arguably more natural, to start from an interaction matrix (or graph) for $n$ players and see where it converges, rather than constructing a specific $n$-player network with a desired limit in mind. This, in a sense, makes the $n$-player game the starting point of the model, rather than the graphon game. This perspective is shared by \cite{bayraktar2022propagation,cui2021learning,gao2020linear}, though these papers impose various restrictions on the graphs and graphon that our main result (Theorem \ref{th:approxEQ-general}) does not need.

Lastly, to illustrate the relative simplicity of our framework, we study in Section \ref{se:LQ} a linear-quadratic model of flocking type similar to \cite{carmona2013mean,lacker2021case}. We explicitly solve the model in terms of a centrality index of a given graphon.

The short Section \ref{se:graphon-notation} introduces the basic notions of kernels and graphons that will be used in the paper. Then, Section \ref{se:mainresults} presents the main results in full detail.

\subsection*{Common notation}
We write $[n] := \{1,\ldots,n\}$ for $n \in \N$.
For a random variable $X$ taking values in a measurable space, we write $\L(X)$ for its law.
For a complete separable metric space $(E,d)$, we write $\M_+(E)$ for the space of nonnegative Borel measures of finite variation, and $\P(E)$ for the sets of probability measures.
We write $\langle \mu,\varphi\rangle  = \int_E\varphi\,d\mu$ for $\mu \in \M_+(E)$ and suitably integrable functions $\varphi$.
We equip $\M_+(E)$ with the usual topology of weak convergence, defined in duality with the space of bounded continuous functions.
This topology is also induced by the bounded-Lipschitz norm (see \cite[Theorem 8.3.2]{bogachev2007measure})
\begin{align}
\|\mu\|_{BL} := \sup\left\{\int_E \varphi \,d\mu : \varphi : E\to\R, \, |\varphi| \le 1, \, \sup_{x \neq y}\frac{|\varphi(x)-\varphi(y)|}{d(x,y)} \le 1 \right\}. \label{bounded-Lipschitz-norm}
\end{align}

We write $C([0,T];E)$ for the space of continuous functions $[0,T] \to E$, always equipped with the supremum distance $(x,x') \mapsto \sup_{t \in [0,T]}d(x_t,x'_t)$. 

We write $\mathrm{Unif}[0,1]$ to denote the uniform (Lebesgue) measure on $[0,1]$. Similarly, $\mathrm{Unif}(I)$ denotes the uniform probability measure on any interval $I$. For a Polish space $E$, let us write also $\Punif([0,1] \times E)$ for the set of Borel probability measures on $[0,1] \times E$ with uniform first marginal. Any $\mu \in \Punif([0,1] \times E)$ admits the disintegration $\mu(du,dx)=du\mu_u(dx)$, with $[0,1] \ni u \mapsto \mu_u \in \P(E)$ being Borel measurable and uniquely defined up to a.e.\ equality. The space $E$ will typically be either $\R^d$ or the path space $\C^d := C([0,T];\R^d)$.

\section{Kernels and graphons} \label{se:graphon-notation}

In this section, we give a brief summary  of the notion of graphon relevant to our work, most importantly introducing (in Section \ref{se:graphon-operators}) its associated operator which will play a central role. We borrow most terminology from Lov\'asz \cite{lovasz2012large}. A \emph{graphon} is typically defined as a symmetric measurable function $[0,1]^2\to[0,1]$. More generally, a \emph{kernel} is any element of $L^1[0,1]^2$, i.e., an integrable Borel-measurable real-valued function of $[0,1]^2$. 

We work with kernels belonging to $L^1_+[0,1]^2$, the set of non-negative elements of $L^1[0,1]^2 $.
We think of $[0,1]$ as indexing a continuum of possible locations or vertices, with $W(u,v)$ representing the (weighted) edge density between them.
 We notably do not require our kernels to be graphons (bounded or symmetric), which brings certain advantages in the examples below. Also, we work with \emph{labeled}  rather than \emph{unlabeled} kernels  \cite[Sections 8.2.1 and 8.2.2]{lovasz2012large}.

For $n \in\N$, the space of $n \times n$ matrices embeds into the space of kernels as follows.
For an $n \times n$ matrix $\xi$, we introduce the associated \emph{step kernel}
\begin{align}
\begin{split}
W_{\xi}(u,v) &:= \xi_{ij}, \quad \text{ for } (u,v) \in I^n_i \times I^n_j, \\ 
\text{where } I^n_i &:= [(i-1)/n,i/n), \ \text{ for } i=1,\ldots,n-1, \ \text{ and } I^n_n := [(n-1)/n,1]. 
\end{split} \label{def:stepgraphon}
\end{align}

\subsection{The cut norm}
Following \cite[Chapter 8.2]{lovasz2012large}, we define the \emph{cut norm} on $L^1[0,1]^2 $ by
\begin{align*}
\|W\|_{\square} := \sup_{S_1,S_2}\bigg| \int_{S_1}\int_{S_2} W(u,v)\,dudv\bigg|,
\end{align*}
where the supremum is over pairs of Borel sets $S_1,S_2 \subseteq [0,1]$. (Technically, this is merely a \emph{seminorm} unless we identify functions which agree a.e.)
The cut norm is clearly weaker than the $L^1$ norm,
\begin{align}
\|W\|_{\square} \le \|W\|_{L^1[0,1]^2 } := \int_0^1\int_0^1|W(u,v)|\,dudv, \label{cutnorm-le-L1norm}
\end{align}
The cut norm is convenient in part because many natural random graph models converge in cut norm but not in $L^1[0,1]^2 $. We provide below two examples where the convergence is well known:

\begin{example} \label{ex:ErdosRenyi}
Let $\xi^n$ be the adjacency matrix of an Erd\H{o}s-R\'enyi random graph $G(n,p_n)$. If $p_n=p$ is fixed as $n\to\infty$, then $W_{\xi^n}$ converges in cut norm to the constant graphon $W \equiv p$.
Allowing unbounded kernels allows one to treat sparser regimes: Instead of assuming $p_n$ to be constant, assume merely that $np_n \to \infty$ as $n\to\infty$. Then $W_{\xi^n/p_n}$ converges in cut norm to the constant graphon $W \equiv 1$. See  \cite[Theorem 2.14(b)]{borgs-chayes-cohn-zhao-I} for a more general result.
\end{example}

\begin{example} \label{ex:sampling}
Given a graphon $W$, i.e. a symmetric and measurable function from $[0,1]^2$ to $[0,1]$,  one can define two natural graphs on the vertex set $[n]$. First, let $U_1,\ldots,U_n \sim \mathrm{Unif}[0,1]$ be independent, and order them $U_{(1)} < \cdots < U_{(n)}$. Then, for $i \ne j$, either
\begin{enumerate}
\item connect vertices $(i,j)$ with probability $W(U_i,U_j)$, or
\item assign weight $W(U_i, U_j)$ to the edge between $(i,j)$.
\end{enumerate}

Note that the latter defines a weighted graph, the former a simple graph. The step kernel associated with the adjacency matrix converges in probability in cut norm to $W$ in either case, and in $L^1$-norm in the latter case.
See  \cite[Theorem 2.14]{borgs-chayes-cohn-zhao-I} for proof, along with related sparse graph constructions for kernels $W$ which are not necessarily bounded.
\end{example}

\subsection{Operators associated with kernels} \label{se:graphon-operators}
To a kernel $W \in L^1[0,1]^2 $ we associate the operator $\textsf{W} : L^\infty[0,1] \to L^1[0,1]$, defined by
\begin{align}
\textsf{W}\varphi(u) := \int_0^1 W(u,v)\varphi(v)\,dv. \label{def:Woperator-function}
\end{align}
The resulting operator norm is equivalent to the cut norm \cite[Lemma 8.11]{lovasz2012large}:
\begin{align}
\|W\|_{\square} \le \|\textsf{W}\|_{\infty \to 1} &\le 4\|W\|_{\square}, \label{ineq:cutnorm-opnorm} \\
\text{where } \|\textsf{W}\|_{\infty\to 1} &:= \sup\{\|\textsf{W}\varphi\|_{L^1[0,1]} : \varphi \in L^\infty[0,1], \, |\varphi| \le 1\}. \label{def:opnorm}
\end{align}
We work most often with the \emph{strong operator topology} for operators on $L^\infty[0,1] \to L^1[0,1]$: We say that a sequence $W_n \in L^1[0,1]^2$ converges in the strong operator topology to $W \in L^1[0,1]^2$ if $\|\textsf{W}_n\varphi-\textsf{W}\varphi\|_{L^1[0,1]} \to 0$ for every $\varphi \in L^\infty[0,1]$.
Convergence in cut norm implies convergence in strong operator topology, by \eqref{ineq:cutnorm-opnorm}.
While the cut norm is the most common in the graphon literature, working more generally with the strong operator topology leads to no increase in difficulty in any of our proofs.

A key object in our paper is a more general operator associated with a kernel $W \in L^1_+[0,1]^2$. Given a Polish space $E$ and a probability measure ${m}$ on $[0,1] \times E$, we define a measure-valued function $\textsf{W}{m} : [0,1] \to \M_+(E)$ by 
\begin{align}
\textsf{W}{m}(u) := \int_{[0,1] \times E} W(u,v) \delta_x \, {m}(dv,dx). \label{def:Woperator-measure}
\end{align}
To be clear, this measure acts on a bounded measurable function $\varphi :E \to \R$ by
\begin{align*}
\langle \textsf{W}{m}(u),\varphi\rangle = \int_{[0,1] \times E} W(u,v) \varphi(x) \, {m}(dv,dx).
\end{align*}
Note that if $W \equiv 1$ then $\textsf{W}{m}(u)$ is exactly the second marginal of ${m}$.

To foreshadow how we will use this operator, think of the measure $\textsf{W}{m}(u)$ as representing a continuous version of the neighborhood empirical measure around a vertex $u$. Indeed, suppose $x_1,\ldots,x_n \in E$ represent state variables of players $1,\ldots,n$, and let $\xi=(\xi_{ij})$ denote an $n \times n$ matrix representing interactions. The influence of the other players on player $i$ is given by the \emph{neighborhood empirical measure}
\begin{align*}
M_i = \frac{1}{n}\sum_{j=1}^n \xi_{ij} \delta_{x_j} \in \M_+(E).
\end{align*}
Suppose $u_1,\ldots,u_n \in [0,1]$ represent labels of the $n$ players, with $u_i \in I^n_i$ for each $i$.
The label-state empirical measure of the entire population is given by
\begin{align*}
M = \frac{1}{n}\sum_{i=1}^n\delta_{(u_i,x_i)} \in \P([0,1] \times E).
\end{align*}
Using the step kernel from \eqref{def:stepgraphon}, the function $\textsf{W}_{\xi}M$ then encodes all of the neighborhood empirical measures in terms of the label-state empirical measure, in the sense that 
\begin{align*}
\textsf{W}_{\xi}M(u_i) &= \frac{1}{n}\sum_{j=1}^n W_{\xi}(u_i,u_j)\delta_{x_j} = \frac{1}{n}\sum_{j=1}^n \xi_{ij} \delta_{x_j} = M_i.
\end{align*}

\begin{remark}
The two operators both denoted $\textsf{W}$, defined in \eqref{def:Woperator-function} to act on real-valued functions and in \eqref{def:Woperator-measure} to act on measures, are not as different as they might at first appear. First, note that the former definition extends readily to functions $\varphi$ with values in suitable vector spaces. Suppose ${m}$ has uniform first marginal, so that by disintegration we may write ${m}(du,dx)=du{m}_u(dx)$. We may then write $\textsf{W}{m}(u) = \int_0^1 W(u,v){m}_v\,dv$, which has the form of \eqref{def:Woperator-function} but with the measure-valued function ${m}_\cdot$ in place of the scalar function $\varphi$.
\end{remark}

\begin{example}[Laplacian matrices] \label{ex:Laplacian}
A natural setting, studied for instance in \cite{delarue2017mean,lacker2021case}, arises from the so-called \emph{random walk Laplacian} of a connected graph on $n$ vertices. Let us write $i \sim j$ if two vertices $i$ and $j$ are neighbors in this graph, and let $d_i$ denote the degree (number of neighbors) of vertex $i$. Then $\xi$ is defined by setting $\xi_{ij}=n/d_i$ if $i\sim j$ and $\xi_{ij}=0$ otherwise. In this case, $M_i=\frac{1}{d_i}\sum_{j \sim i}\delta_{x_j}$ is the uniform measure over the states of the neighbors of $i$.
\end{example}

\section{Main results}\label{se:mainresults}

In this section we define precisely the $n$-player and graphon game models. 
The following assumptions are in force throughout the paper.

\begin{assumption} 
We are given dimensions $d,d_0 \in \N$, a time horizon $T > 0$, a compact metric space $A$ representing the set of actions, and bounded continuous functions
\begin{equation*}
    \begin{array}{ll}
        b : [0,T] \times \R^d \times A \to \R^d  & \qquad  \sigma :  [0,T] \times \R^d \to \R^{d\times d_0},\\
        f : [0,T] \times \R^d \times \M_+(\R^d) \times A \to \R  & \qquad g : \R^d \times \M_+(\R^d) \to \R.
    \end{array}
\end{equation*}
Assume that $\sigma$ is Lipschitz and that $\sigma\sigma^\top$ is uniformly nondegenerate, i.e., bounded from below in semidefinite order by a positive constant times the identity matrix.
Assume further that for each $(t,x,m) \in [0,T] \times \R^d \times \M_+(\R^d)$ the following set  is convex:
\begin{align}
\{(b(t,x,a),z) : a \in A, \, z \le f(t,x,m,a)\} \subset \R^d \times \R. \label{asmp:convex}
\end{align}
Finally, we are given an initial distribution $\lambda \in \Punif([0,1] \times \R^d)$.
\end{assumption}

These assumptions can certainly be generalized, particularly the boundedness. We prefer to minimize technicalities in order to focus on the new features of the graphon setting.
The final convexity assumption is common in the control literature; it holds when $A$ is a convex subset of a vector space, $b$ is affine in $a$, and $f$ is concave in $a$, which includes in particular the setting of \emph{relaxed controls} to which one can always lift the problem if the convexity assumption is not initially satisfied \cite{lacker2015mean}.

The most notable restriction is that we do not include any interaction term within the functions $b$ or $\sigma$. This significantly simplifies the existence theorem and the approximate equilibrium construction. The former would easily generalize, but the latter would require a satisfactory limit theory for graphon-based interacting SDEs. Such a limit theory is a significant undertaking in its own right and has seen only very recent development, so far only for scalar interactions. By excluding interactions from $(b,\sigma)$, we avoid this separate issue and focus more on the game-theoretic aspects of graphon models.

We work with Markovian controls throughout the paper, but the framework adapts easily to different kinds of controls, such as open-loop.

\subsection{Finite games} \label{se:finitegame}

Let $n\in \N$ denote the number of players. Each player may choose a control from $\A_n$, the set of measurable functions from $[0,T] \times (\R^d)^n \to A$. For any vector of controls $\bm\alpha =(\alpha_1,\ldots,\alpha_n) \in \A_n^n$, there exists a unique solution $\bm{X}^n=(X^{n,1},\ldots,X^{n,n})$ of the  SDE system
\begin{align*}
dX^{n,i}_t = b(t, X^{n,i}_t, \alpha_i(t,\bm{X}^n_t)) dt + \sigma(t, X^{n,i}_t) dB^i_t, \qquad X^{n,i}_0=x^{n,i}_0,
\end{align*}
where $B^1,\ldots,B^n$ are independent $d_0$-dimensional Brownian motions, and $x^{n,i}_0$ are given initial conditions.

The  boundedness of $b$ and Lipschitz continuity of $\sigma$ ensure that this SDE system admits a unique strong solution \cite[Theorem 1]{veretennikov1981strong}.

Let $\xi^n=(\xi^n_{ij})$ denote an $n \times n$ matrix  with nonnegative entries, called the \emph{interaction matrix}. Throughout this paper we will assume that $\xi^n_{ii}=0$ for all $i$; if $\xi^n$ is the adjacency matrix of a (weighted) graph, this is equivalent to assuming that there are no self-loops. This assumption is natural and simplifies the exposition, but it is not hard to generalize.
A key role is played by the \emph{neighborhood empirical measures}, defined for each player $i\in [n]$ by
\begin{align}
M^{n,i}_t = \frac{1}{n}\sum_{j=1}^n \xi^n_{ij} \delta_{X^{n,j}_t}, \label{nbd-empiricalmeasure-nplayers}
\end{align}
which is a random element of $\M_+(\R^d)$.
For $\bm\alpha =(\alpha_1,\ldots,\alpha_n) \in \A_n^n$, the objective function of each player $i \in [n]$ is defined by 
\begin{align}\label{def:J_i}
J_i(\bm{\alpha}) &:= \E\left[ \int_0^T f(t,X^{n,i}_t,M^{n,i}_t,\alpha_i(t,\bm{X}^n_t))\,dt + g(X^{n,i}_T,M^{n,i}_T)\right].
\end{align}
For $\bm{\epsilon}=(\epsilon_1,\ldots,\epsilon_n) \in [0,\infty)^n$, an \emph{$\bm{\epsilon}$-Nash equilibrium} is defined as any $\bm\alpha =(\alpha_1,\ldots,\alpha_n) \in \A_n^n$ satisfying for all $i \in [n]$
\begin{align*}
J_i(\bm{\alpha}) &\ge \sup_{\beta \in \A_n} J_i(\alpha_1,\ldots,\alpha_{i-1},\beta,\alpha_{i+1},\ldots,\alpha_n) - \epsilon_i.
\end{align*}

We will not state any theorems about $n$-player games until Section \ref{subse:main_res_approx_eq}, but it will inform our definition of the appropriate graphon model in the following section.

\subsection{Graphon games} \label{se:graphongame}

For a kernel $W \in L^1_+[0,1]^2 $, we define the \emph{graphon game} associated with $W$ as follows. Let $\Aunif$ denote the set of measurable functions $[0,T] \times [0,1] \times \R^d \to A$. Let $(\Omega,\F,\FF,\PP)$ be a filtered probability space supporting a $d_0$-dimensional $\FF$-Brownian motion $B$ and $\F_0$-measurable random variables $U$ and $X_0$ taking values in $[0,1]$ and $\R^d$, respectively. The given joint law of $(U,X_0)$ is denoted $\lambda$, and its first marginal is assumed to be uniform; that is, $U \sim$ Unif$[0,1]$.  For $\alpha \in \Aunif$, the state process $X$ is the unique solution of the SDE
\begin{align}
dX^\alpha_t = b(t,X^\alpha_t,\alpha(t,U,X^\alpha_t))dt + \sigma(t,X^\alpha_t)dB_t, \quad X^\alpha_0=X_0. \label{def:SDEgraphongame}
\end{align}
Strong well-posedness of this SDE follows easily from \cite[Theorem 1]{veretennikov1981strong}, under our standing assumptions.
Recall in the following the meaning of $\textsf{W}\mu_t$, defined in \eqref{def:Woperator-measure}, as well as the notation $\Punif([0,1] \times E)$ for measures on $[0,1] \times E$ with uniform first marginal.

Now, to define our notion of equilibrium, suppose we are given a measure flow $\mu_\cdot =(\mu_t)_{t \in [0,T]} \in C([0,T];\Punif([0,1] \times \R^d))$, representing the label-state joint distribution at each time.
In response to this given $\mu_\cdot$, the objective of a typical player is to choose $\alpha \in \Aunif$ to maximize
\begin{equation*}
J_W(\mu_\cdot,\alpha) := \E \left[\int_0^T f(t, X^\alpha_t, \textsf{W}\mu_t(U), \alpha(t,U,X^\alpha_t)) dt + g(X^\alpha_T, \textsf{W}\mu_T(U)) \right].
\end{equation*}
The measure $\textsf{W}\mu_t(U)$ here is the natural graphon analogue of the neighborhood empirical measure, as discussed in Section \ref{se:graphon-operators}, when a player is given the uniformly random label $U$.

\begin{definition}
We say that $\mu_\cdot\in C([0,T];\Punif([0,1] \times \R^d))$ is a \emph{(Markovian) $W$-equilibrium} (or a \emph{graphon equilibrium} when $W$ is understood) if there exists $\alpha^* \in \Aunif$ satisfying
\begin{align*} 
J_W(\mu_\cdot,\alpha^*) = \sup_{\alpha \in \Aunif}J_W(\mu_\cdot,\alpha), \quad \text{and} \quad \mu_t=\L(U,X^{\alpha^*}_t), \ \forall t \in [0,T].
\end{align*}
Any such $\alpha^*$ is called an \emph{equilibrium control for $\mu_\cdot$}.
\end{definition}

We might describe this fixed point problem loosely but compactly as follows:
\begin{equation}
    \left\{ \begin{array}{rl}
        \alpha^* \in \arg\!\max_{\alpha} &\E \left[\int_0^T f(t, X^\alpha_t, \textsf{W}\mu_t(U), \alpha_t) dt + g(X^\alpha_T, \textsf{W}\mu_T(U)) \right] \\
        \mbox{s.t.} & dX^\alpha_t = b(t,X^\alpha_t,\alpha_t) dt + \sigma(t, X^{\alpha}_t) dB_t, \\
        & \mu_t = \L(U, X^{\alpha}_t),  \ \ (U, X_0) \sim \lambda.
    \end{array}
    \right. \label{def:graphonEQ-compact}
\end{equation}

For comparison, we also state the classical definition of a mean field game equilibrium, in the case where there is no graphon present (or $W \equiv 1$). Note that the space $\A_1$ of measurable functions $[0,T] \times \R^d \to A$ may be identified with the subspace of $\Aunif$ consisting of controls that do not depend on the uniform variable $U$, i.e., functions of the form $\alpha(t,u,x)=\tilde\alpha(t,x)$. We say that $\nu_\cdot \in C([0,T];\P(\R^d))$ is a \emph{(Markovian) mean field equilibrium} if there exists $\alpha^* \in \A_1$ satisfying
\begin{align*} 
J_1(\nu_\cdot,\alpha^*) = \sup_{\alpha \in \A_1}J_1(\nu_\cdot,\alpha), \quad \text{and} \quad \nu_t=\L(X^{\alpha^*}_t) \ \ \forall t \in [0,T],
\end{align*}
where we define
\begin{align*}
J_1(\nu_\cdot,\alpha) := \E \left[\int_0^T f(t, X^\alpha_t, \nu_t , \alpha(t, X^\alpha_t)) dt + g(X_T, \nu_T) \right].
\end{align*}
When $W\equiv 1$, recall that $\textsf{W}{m}$ reduces to the second marginal of ${m}(dv,dx)$; it follows  that if $\mu_\cdot$ is a $W$-equilibrium then the second marginals form a mean field equilibrium. The converse is true but somewhat more subtle, because controls for mean field equilibria are allowed to depend on the auxiliary random variable $U$. See Proposition \ref{prop:constdegree} for a more general relationship between these two equilibrium concepts.

\subsection{Existence and uniqueness of equilibria} \label{se:existence&uniqueness}

Recall in the following that we are always working under the standing assumptions stated at the beginning of Section \ref{se:mainresults}.
The following is proven in Section \ref{se:existence_GMFG_sol}, following the strategy of \cite{lacker2015mean}.

\begin{theorem} \label{th:existence}
Let $W \in L^1_+[0,1]^2 $. Then there exists a $W$-equilibrium.
\end{theorem}

For certain $W$, a mean field equilibrium can  be identified with a graphon equilibrium. This is clear when $W \equiv 1$, as noted above, but in fact holds more generally:

\begin{proposition} \label{prop:constdegree}
Let $W \in L^1_+[0,1]^2 $. Assume that 
\begin{align}
\int_0^1 W(u,v)\,dv = 1, \quad a.e.\ u \in [0,1]. \label{def:const-degree}
\end{align}
Suppose $\nu_\cdot \in C([0,T];\P(\R^d))$ is a mean field equilibrium, and let $\alpha^* \in \A_1$ be an equilibrium control for $\nu_\cdot$. Define $\mu_t = \mathrm{Unif}[0,1] \times \nu_t$. Then $\mu_\cdot=(\mu_t)_{t \in [0,T]}$ is a $W$-equilibrium, and  $(t,u,x) \mapsto \alpha^*(t,x)$ is an equilibrium control for $\mu_\cdot$.
\end{proposition}

The condition \eqref{def:const-degree} can be interpreted as saying that the graphon $W$ has \emph{constant out-degree}, or simply \emph{constant degree} if $W$ is assumed symmetric. A similar principle appeared in the uncontrolled setting in \cite[Corollary 2.4]{coppini2021note}.

\begin{example}
Let us revisit Example \ref{ex:Laplacian}, where $G_n$ is a simple connected graph on vertex set $[n]$, and $\xi^n_{ij}=(n/d_i)1_{\{i \sim j\}}$.
The neighborhood empirical measures become
\begin{align*}
M^{n,i}_t = \frac{1}{n}\sum_{j=1}^n \xi^n_{ij} \delta_{X^{n,j}_t} = \frac{1}{d_i}\sum_{j \sim i} \delta_{X^{n,j}_t}.
\end{align*}
This models a scenario in which players interact symmetrically with their neighbors in the underlying graph $G_n$, as in \cite{delarue2017mean,lacker2021case}.
It is not clear if there is a simple (e.g., degree-based) characterization of the situations where $W_{\xi^n}$ converges in the strong operator topology (or in cut norm). However, if a limit $W_{\xi^n} \to W$ does exist, then $W$ must satisfy the constant-degree condition of Proposition \ref{prop:constdegree}. Indeed, for each $i \in [n]$ and each $u \in I^n_i$ we have
\begin{align*}
\int_0^1 W_{\xi^n}(u,v)dv &= \sum_{j=1}^n \int_{I^n_j} W_{\xi^n}(u,v)dv = \sum_{j=1}^n \frac{1}{n} \xi^n_{ij} = \sum_{j=1}^n \frac{1}{d_i} 1_{\{i \sim j\}} = 1,
\end{align*}
and the left-hand side, as a function of $u$, converges in $L^1[0,1]$ to $\int_0^1 W(u,v)dv$.
\end{example}

We can further show uniqueness of the equilibrium, under an additional assumption adapted from the classical Lasry-Lions monotonicity condition:

\begin{proposition}\label{prop:unique_W_eq}
In addition to the standing assumptions of Section \ref{se:mainresults}, assume the following:
\begin{enumerate}
\item Separable $f$: There exist two functions $f_1, f_2$ such that  
    \begin{equation*}
        f(t,x,m,a) = f_1(t,x,a) + f_2(t,x,m),
    \end{equation*}
    \item Unique optimal controls: For each $\mu \in C([0,T]; \Punif([0,1]\times \R^d)$, the supremum in $\sup_{\alpha \in \A_U}J_W(\mu,\alpha)$ is attained uniquely (up to Lebesgue a.e.\ equality).
    \item Monotonicity: for each ${m}_1, {m}_2 \in \Punif([0,1] \times \R^d)$ and $t \in [0,T]$, we have 
    \begin{align}
        \int_{[0,1] \times \R^d} \big(g(x, \text{\rm\textsf{W}}{m}_1(u)) - g(x, \text{\rm\textsf{W}}{m}_2(u))\big)&({m}_1 - {m}_2)(du, dx) \le 0 \label{def:monotonicity} \\
        \int_{[0,1] \times \R^d} \big(f_2(t,x, \text{\rm\textsf{W}}{m}_1(u)) - f_2(t,x, \text{\rm\textsf{W}}{m}_2(u))\big)&({m}_1 - {m}_2)(du,dx) \le 0 . \nonumber
    \end{align}
\end{enumerate}
Then there exists a unique $W$-equilibrium.
\end{proposition}

The proof is given in Section \ref{subse:proof_unique}, along with a couple of noteworthy examples of functions $g$ satisfying \eqref{def:monotonicity} (see Remark \ref{re:monotonicity-examples}). 
The proof follows by reducing the graphon game to  a classical mean field game, explained in more detail in Section \ref{se:PDE} below.

\subsection{Approximate equilibria}\label{subse:main_res_approx_eq}

Throughout this section, we are given $W \in L^1_+[0,1]^2$, and we let $\mu_\cdot \in C([0,T];\Punif([0,1] \times \R^d))$ denote a $W$-equilibrium and $\alpha^*$ an equilibrium control for $\mu_\cdot$. Also, as in Section \ref{se:finitegame}, we are given an arbitrary $n \times n$ matrix $\xi^n$ with positive entries and zeros on the diagonal, $\xi^n_{ii}=0$. We define the step kernel $W_{\xi^n}$ as in \eqref{def:stepgraphon}.

In this section we explain how the graphon game defined in Section \ref{se:graphongame} gives rise to approximate equilibria for the finite game defined in Section \ref{se:finitegame}, when the underlying kernels $W_{\xi^n}$ from \eqref{def:stepgraphon} converge in a suitable sense to the kernel $W$.
To provide context for the following results, let us briefly recall the analogous construction in mean field game theory. If $\widehat\alpha \in \A_1$ denotes a mean field equilibrium control, then players $i \in [n]$ in the $n$-player game are assigned the controls $\alpha^n_i(t,x_1,\ldots,x_n) = \widehat\alpha(t,x_i)$. The vector $(\alpha^n_1,\ldots,\alpha^n_n)$ is then shown to constitute an $\epsilon^n$-equilibrium, where $\epsilon^n\to 0$. This strategy dates back to the earliest work on mean field games \cite{huang2006large}, and see \cite[Section 6.1]{carmona-delarue-book} or \cite[Section 2.4]{lacker2020convergence} for the closed-loop case.

This strategy requires several adaptations in the present context. First, because players are not exchangeable, we may have a different error $\epsilon^n_i$ for each player. Moreover, different modes of convergence to zero can make sense in different contexts, such as $\frac{1}{n}\sum_{i=1}^n\epsilon^n_i \to 0$ or $\max_{i \in [n]}\epsilon^n_i \to 0$. This is also highlighted in our case study \cite{lacker2021case}.

A second and more delicate point in our setting is in how to deal with labels.
A $W$-equilibrium control $\alpha^* \in \Aunif$ depends on an additional $\mathrm{Unif}[0,1]$ variable, which we have interpreted as the label (or vertex) of the player. In order to apply this control $\alpha^*$ in the $n$-player game, we must specify which labels to assign to each player. In the definition of the step kernel $W_{\xi^n}$, the player $i$ in the $n$-player game is associated with the interval $I^n_i$ defined in \eqref{def:stepgraphon}, and it thus makes sense to choose for player $i$ some label $u^n_i \in I^n_i$. We then assign to player $i$ the control
\begin{align}\label{eq:approx_equilibrium_startegy}
\alpha_i^{n,u^n_i}(t,x_1,\ldots,x_n) := \alpha^*(t,u^n_i,x_i).
\end{align}
The error $\epsilon^n_i(\bm{u}^n)$ then depends additionally on the choice of labels $\bm{u}^n=(u^n_1,\ldots,u^n_n)$, and the question again arises as to the sense in which we can expect these errors to vanish as $n\to\infty$. In general, we only expect these errors to vanish in probability, with respect to a random choice of $\bm{u}^n$, but we will see that stronger continuity assumptions allow us to strengthen the convergence to be (essentially) uniform in the choice of $\bm{u}^n$.

Let us define precisely the function $\epsilon^n_i : [0,1]^n \to [0,\infty)$. Fix $\bm{u}^n=(u^n_1,\ldots,u^n_n) \in [0,1]^n$ in this paragraph. Using the construction \eqref{eq:approx_equilibrium_startegy}, define $\bm{\alpha}^{n,\bm{u}^n}=(\alpha_1^{n,u^n_1},\ldots,\alpha_n^{n,u^n_n}) \in \A_n^n$. Recall that $\lambda(du,dx)=du\lambda_u(dx)$ denotes the given joint law of $(U,X_0)$ in the graphon game. Consider the $n$-player game as described in Section \ref{se:finitegame}, with initial conditions $(X^{n,i}_0)_{i=1}^n$ chosen independently with $X^{n,i}_0 \sim \lambda_{u^n_i}$. With this choice of initialization, we finally define the nonnegative number
\begin{align*}
\epsilon^n_i(\bm{u}^n) := \sup_{\beta \in \A_n} J_i(\alpha_1^{n,u^n_1},\ldots,\alpha_{i-1}^{n,u^n_{i-1}},\beta,\alpha_{i+1}^{n,u^n_{i+1}},\ldots,\alpha_n^{n,u^n_n}) - J_i(\bm{\alpha}^{n,\bm{u}^n}).
\end{align*}
By definition, $\bm{\alpha}^{n,\bm{u}^n}$ is a $\bm\epsilon^n(\bm{u}^n)$-equilibrium, where $\bm\epsilon^n(\bm{u}^n)=(\epsilon^n_1(\bm{u}^n),\ldots,\epsilon^n_n(\bm{u}^n))$.
This definition makes sense only if we prespecify a version of the disintegration $u \mapsto \lambda_u$, and otherwise we should understand $\bm\epsilon^n(\bm{u}^n)$ to be uniquely defined only up to $\bm{u}^n$-a.e.\ equality.

We first show in full generality that $\bm\epsilon^n\to 0$ in an averaged sense. Recall from Section \ref{se:graphon-operators} the definition of the strong operator topology, for operators from $L^\infty[0,1]$ to $L^1[0,1]$, and recall that convergence in this topology is implied by convergence in cut norm.

\begin{theorem}[General kernel] \label{th:approxEQ-general}
Assume the disintegration $u \mapsto \lambda_u$ admits a version such that $\{\lambda_u : u \in [0,1]\}$ is tight.
Assume $W_{\xi^n}$ converges in the strong operator topology to $W$, and also
\begin{align}
\lim_{n\to\infty}\frac{1}{n^3}\sum_{i,j=1}^n (\xi^n_{ij})^2 = 0. \label{asmp:generalprinciple-A}
\end{align}
Then, if for each $n \in \N$, $(U^n_1,\ldots,U^n_n)$ are independent with $U^n_i \sim \mathrm{Unif}(I^n_i)$,
\begin{align*}
\lim_{n\to\infty}\frac{1}{n} \sum_{i=1}^n \E \left[\epsilon^n_i(U^n_1, \ldots, U^n_n)\right] = 0.
\end{align*}
\end{theorem}

The proof of Theorem \ref{th:approxEQ-general} is given in Section \ref{se:approxEQ}, along with the proofs of the two other theorems of this section. The bulk of the analysis is presented first in Section \ref{sec:conv_emp_meas}, in a more general setting that clarifies the key points.

\begin{remark}
The assumption \ref{asmp:generalprinciple-A} is very mild. It holds trivially if $|\xi^n_{ij}|$ are uniformly bounded. If $\xi^n$ is $1/p_n$ times the adjacency matrix of the Erd\H{o}s-R\'enyi graph $G(n,p_n)$, then \ref{asmp:generalprinciple-A} is easily shown to hold in probability, when $np_n \to \infty$.
\end{remark}

\begin{remark} \label{re:eps-general}
We have assumed $f$ and $g$ to be bounded, which means $\epsilon^n_i$ are uniformly bounded. Hence, the conclusion of Theorem \ref{th:approxEQ-general} is equivalent to saying that $\epsilon^n_{I_n}(U^n_1, \ldots, U^n_n) \to 0$ in probability, where $I_n \sim$ Unif$([n])$.
In other words, for randomly assigned labels from $I^n_1 \times \cdots \times I^n_n$, and for a randomly chosen player from $[n]$, the error is small. Note that this does not rule out the possibility that certain players and label assignments have large errors $\epsilon^n_i$, but the fraction of such players and label assignments is negligible.
\end{remark}

Our next result strengthens the mode of convergence, at the price of requiring stronger continuity assumptions, both on the graphon and on the optimal state process. Recall in the following that $\C^d=C([0,T];\R^d)$, and $(U,X^{\alpha^*})$ is defined as in Section \ref{se:graphongame}.

\begin{theorem}[Continuous kernel] \label{th:approxEQ-continuous}
Assume the following:
\begin{enumerate}
\item The map $[0,1] \ni u \mapsto W(u,v)dv \in \M_+([0,1])$ is continuous.
\item The disintegration $[0,1] \ni u \mapsto \L(X^{\alpha^*}\,|\,\,U=u) \in \P(\C^d)$ admits a continuous version. 
\end{enumerate}
Assume that \eqref{asmp:generalprinciple-A} holds, and that $W_{\xi^n}$ converges in the strong operator topology to $W$.
Then 
\begin{align*}
\lim_{n\to\infty}\underset{\bm{u}^n \in I^n_1 \times \cdots \times I^n_n}{\rm{ess} \sup} \ \frac{1}{n} \sum_{i=1}^n \epsilon^n_i(\bm{u}^n) = 0.
\end{align*}
Moreover, if (1) holds, then (2) holds under the following additional conditions:
\begin{enumerate}
\item[(2a)] The disintegration $[0,1] \ni u \mapsto \lambda_u \in \P(\R^d)$ admits a continuous version.
\item[(2b)] $A$ is a compact convex subset of $\R^k$ for some $k \in \N$.
\item[(2c)] $\sigma(t,x)=\sigma$ is constant.
\item[(2d)] For each $(t,x)$, $a\mapsto b(t,x,a)$ is affine, and $a \mapsto f(t,x,m,a)$ is strictly concave.
\end{enumerate}
\end{theorem}

To be clear, the two continuity assumptions in Theorem \ref{th:approxEQ-continuous} mean that $\int_0^1 W(u,v)h(v)\,dv$ and $\E[\varphi(X^{\alpha^*})\,|\,U=u]$ depend continuously on $u$, for all bounded continuous real-valued functions $h$ and $\varphi$ on $[0,1]$ and $\C^d$, respectively.
In particular, (1) is true if the function $W : [0,1]^2 \to \R$ is itself continuous.
These continuity assumptions allow a finer pointwise control over quantities derived from the graphon, ensuring for instance that the quantities  $\int_0^1 W(u^n_i,v)h(v)\,dv$ and $\E\int_0^1 W(U^n_i,v)h(v)\,dv$ are close, uniformly in the choice of $u^n_i \in I^n_i$, with again $U^n_i \sim$ Unif($I^n_i$).
Stronger continuity assumptions on $W$ were used in \cite{bayraktar2020graphon,bayraktar2022propagation,tangpi2022optimal}.

\begin{remark} \label{re:continuous-disintegration1}
The assumption (2) in Theorem \ref{th:approxEQ-continuous} can be difficult to check, which is why we provide the more tractable sufficient conditions (2a--d). But (2) is actually automatic in the context of Proposition \ref{prop:constdegree}, as $\L(X^{\alpha^*}\,|\,\,U=u)=\L(X^{\alpha^*})$ is constant in $u$.
For an alternative sufficient condition, it is not hard to show that if (1,2a) hold, and if the control $\alpha^*(t,u,x)$ depends continuously on $(u,x)$ for each $t$, then (2) holds.
\end{remark}

\begin{remark} \label{re:eps-cont}
Analogously to Remark \ref{re:eps-general}, the conclusion of Theorem \ref{th:approxEQ-continuous} is equivalent to the following. For every $\epsilon,\delta \in (0,1)$, it holds for sufficiently large $n$ that
\begin{align*}
|\{i \in [n] : \epsilon^n_i(\bm{u}^n) > \epsilon \}| \le n \delta, \quad \text{for a.e. } \bm{u}^n \in I^n_1\times\cdots\times I^n_n.
\end{align*}
In other words, for large enough $n$ and for a.e.\ choice of labels, we have an $(\epsilon,\delta)$-equilibrium in the sense of \cite{carmona2004nash} (used also in \cite{cui2021learning}): no more than a fraction of $\delta$ of the players are further than $\epsilon$ from optimality.
\end{remark}

\begin{remark} \label{re:MFcase}
Our approximate equilibrium results can be combined with Proposition \ref{prop:constdegree} to yield interesting results on the ``universality" of the mean field game approximation.  If $W \in L^1_+[0,1]^2 $ satisfies \eqref{def:const-degree}, and if $W_{\xi^n} \to W$ in the strong operator topology, then a mean field equilibrium (as opposed to a graphon equilibrium) can be used in Theorem \ref{th:approxEQ-general} to construct approximate equilibria for the $n$-player games.
This justifies the intuition mentioned in the introduction, that the usual MFG approximation remains valid for sufficiently \emph{dense} and \emph{approximately regular} networks.
Note as in Remark \ref{re:continuous-disintegration1} that condition (2) of Theorem \ref{th:approxEQ-continuous} holds automatically in this case; hence, if also $u \mapsto W(u,v)dv \in \M_+([0,1])$ is continuous (e.g., if $W \equiv 1$), then we can also apply Theorem \ref{th:approxEQ-continuous} as well.
\end{remark}

Our final result on approximate equilibria deals with the case where the interaction matrix $\xi^n$ is the weighted adjacency matrix obtained by sampling from the graphon $W$ in a standard manner, as in Example \ref{ex:sampling}(b).

\begin{theorem}[Sampling kernel] \label{th:approxEQ-sampling}
Let $W \in L^1_+[0,1]^2$ be bounded. 
Assume the disintegration $u \mapsto \lambda_u$ admits a version such that $\{\lambda_u : u \in [0,1]\}$ is tight. Then the following holds, for almost every choice of $(u_i)_{i\in\N} \in [0,1]^\infty$, where $[0,1]^\infty$ is equipped with the infinite product measure $(\rm{Unif}[0,1])^{\infty}$: Set  $\xi^n_{ij}=W(u_i,u_j)1_{i \neq j}$ for $i,j \in [n]$ in the $n$-player game. Then 
\begin{align*}
\lim_{n\to\infty}\max_{i \in [n]}\epsilon^n_i(u_1,\ldots,u_n) \to 0.
\end{align*}
\end{theorem}

\begin{remark}
Our connection between the initial conditions $(X^{n,i}_0)_{i=1}^n$ and the initial distribution $\lambda$ covers many natural cases. If $(X^{n,i}_0)$ are taken to be i.i.d.\ $\sim \lambda \in \P(\R^d)$, as is common in the MFG literature, then we can simply choose $\lambda(du,dx)=du\lambda(dx)$. 

In general, the initial conditions $X^{n,i}_0 \sim \lambda_{u^n_i}$ may be different for each player, though we do still require them to be independent.
For another example, a player with label $u$ could have a non-random initial position $h(u)$, for some measurable function $h : [0,1] \to \R^d$, in which case the natural choice is $\lambda(du,dx)=du\delta_{h(u)}(dx)$.

It is natural to expect more general results to be possible, in which we assume merely that the initial empirical measure $\frac{1}{n}\sum_{i=1}^n \delta_{(i/n,X^{n,i}_0)}$ converges weakly to $\lambda$.
\end{remark}

\begin{remark}
Another approach to justifying our graphon game formulation would be by studying the \emph{convergence problem}, i.e., the problem of analyzing the $n\to\infty$ behavior of the \emph{true} $n$-player equilibria rather than constructing specific \emph{approximate} equilibria. We do not address this problem in this paper, which was was already a difficult problem in mean field game theory \cite{cardaliaguet2019master,lacker2020convergence}, though we mention the very recent papers \cite{bayraktar2022propagation,tangpi2022optimal} which obtain first results on the convergence problem for graphon games.
\end{remark}

\subsection{Graphon games as mean field games, and their PDE formulation} \label{se:PDE}

This section contains no theorems but illustrates how to recast the graphon equilibrium problem of Section \ref{se:graphongame} as a classical mean field game. The point is simply to view the ``label" variable as a state variable with trivial dynamics. For $\alpha \in \Aunif$, 
the $(d+1)$-dimensional process $\overline{X}^\alpha=(U,X^\alpha)$ is the unique solution of the SDE
\begin{align}\label{eq:full_process_UX_dynamics}
d\overline{X}^\alpha_t = \overline{b}(t,\overline{X}^\alpha_t,\alpha(t,\overline{X}^\alpha_t))\,dt + \overline{\sigma}(t,\overline{X}^\alpha_t)\,dB_t,
\end{align}
where $\overline{b} : [0,T] \times \R^{d+1} \times A \to \R^{d+1}$ and $\overline{\sigma} : [0,T] \times \R^{d+1} \to \R^{(d+1) \times d_0}$ are defined by
\begin{align*}
\overline{b}(t,\overline{x},a) &= \begin{pmatrix}
0 \\ b(t,x,a)
\end{pmatrix}, \quad \overline{\sigma}(t,\overline{x}) = \begin{pmatrix}
0_{d_0}^\top \\ \sigma(t,x,a) 
\end{pmatrix}.
\end{align*}
where we write $\overline{x}=(u,x)$ for a generic element of $\R^{d+1} \cong \R \times \R^d$.
That is, the vector $\overline{b}$ and matrix $\overline{\sigma}$ simply append an additional zero row.
Similarly, define $\overline{f} : [0,T] \times \R^{d+1} \times \P(\R^{d+1}) \times A \to \R$ and $\overline{g} : \R^{d+1} \times \P(\R^{d+1}) \to \R$ by
\begin{align*}
\overline{f}(t,\overline{x},{m},a) &= f(t,x,\textsf{W}{m}(u),a), \quad \overline{g}(\overline{x},{m}) = g(x,\textsf{W}{m}(u)).
\end{align*}
(Define $\overline{f}$ and $\overline{g}$ arbitrarily when $u \notin [0,1]$.)
The graphon equilibrium problem is then nothing but the standard mean field game problem associated with the new coefficients $(\overline{b},\overline{\sigma},\overline{f},\overline{g})$. Indeed, a graphon equilibrium is a measure flow $(\mu_t)_{t \in [0,T]}$ such that there exists $\alpha^*\in\Aunif$ satisfying $\mu_t=\L(\overline{X}^{\alpha^*}_t)$ for all $t \in [0,T]$ as well as
\begin{align*}
\E & \left[\int_0^T\overline{f}(t,\overline{X}^\alpha_t,\mu_t,\alpha(t,\overline{X}^\alpha_t))\,dt + \overline{g}(\overline{X}^\alpha_T,\mu_T)\right] \\
	&= \sup_{\beta \in \Aunif} \E\left[\int_0^T\overline{f}(t,\overline{X}^\beta_t,\mu_t,\beta(t,\overline{X}^\beta_t))\,dt + \overline{g}(\overline{X}^\beta_T,\mu_T)\right].
\end{align*}

It must be stressed that recasting the graphon equilibrium problem as a classical mean field game in this manner does not significantly simplify its analysis (except in the proof of uniqueness, Proposition \ref{prop:unique_W_eq}). There are several reasons that existing theory cannot be applied directly in this framework.
\begin{itemize}
\item The kernel $W : [0,1]^2 \to \R$ is not a continuous function in general. It is in some cases, but in many interesting cases it is not (e.g., the stochastic block model). If $W$ is discontinuous, then $\textsf{W}\mu_t(u)$ is discontinuous in $u$, and thus the objective functions $\overline{f}$ and $\overline{g}$ are discontinuous functions of the state variable $\overline{x}$.
\item In the analysis of approximate equilibria, the natural $n$-player game of Section \ref{se:finitegame} is not equivalent the one obtained by plugging empirical measures into the objective functions $(\overline{f},\overline{g})$. The graphon is different in the $n$-player game, being $W_{\xi^n}$ instead of $W$, and this makes our convergence analysis more difficult.

\item The diffusion matrix $\overline{\sigma}\overline{\sigma}^\top$ of the $(d+1)$-dimensional process $\overline{X}$ is always degenerate, even if that of the original $d$-dimensional state process $X$ is not.
\end{itemize}

Although it does not help with our analysis, recasting the graphon model as a mean field game does reveal what the appropriate PDE formulation should be, in the spirit of Lasry-Lions \cite{lasry-lions}. (Similarly, an FBSDE formulation in the spirit of Carmona-Delarue \cite{carmona-delarue,carmona-delarue-book} is possible as well, but we omit it here.)
Indeed, taking $\sigma$ to be the identity matrix for simplicity, the value function $v(t,u,x)$ and density flow $\mu(t,u,x)$ should (formally) obey the PDE system
\begin{align*}
\begin{cases}
&0 = \partial_t v(t,u,x) + \sup_{a \in A}\big[ b(t,x,a) \cdot \nabla_x v(t,u,x) + f(t,x,\textsf{W}\mu_t(u), a)\big] + \frac12 \Delta_x v(t,u,x)  \\
&\partial_t \mu(t,u,x) = - \mbox{div}_x\big(b(t,x,\widehat\alpha(t,u,x))\mu(t,u,x)\big) + \frac12 \Delta_x \mu(t,u,x)   \\
&\text{where } \ \widehat\alpha(t,u,x) = \arg\!\max_{a \in A}\big[ b(t,x,a) \cdot \nabla_x v(t,u,x) + f(t,x,\textsf{W}\mu_t(u), a)\big], \\
&\text{and } \ v(T,u,x)=g(x,\textsf{W}\mu_T(u)), \quad \mu_0=\lambda.
\end{cases}
\end{align*}

Notably, there are no derivatives with respect to $u$. We will not claim to perform any rigorous analysis of this PDE system. However, it is worth noting that a verification theorem for classical solutions only requires $v$ to be once differentiable in $t$ and twice in $x$, and no differentiability with respect to $u$ is needed. This observation will be used implicitly in our linear-quadratic example in Section \ref{se:LQ}.
Lastly, we mention that the above system of PDEs could be formally interpreted as a continuum of conditional measure flows $((t,x) \mapsto \mu(t,u,x))_{u \in [0,1]}$, which is similar in spirit to the PDE systems discussed in \cite{caines2018graphon}.

\subsection{Organization of the paper}
The remaining sections give the proofs of the main theorems, with the exception of Section \ref{se:LQ} which works out a linear-quadratic example.
Section \ref{se:existence_GMFG_sol} proves existence and uniqueness as stated in Section \ref{se:existence&uniqueness}, and may be read independently of Sections \ref{se:dependence-on-U}--\ref{se:approxEQ} which deal with approximate equilibria. Similarly, the linear-quadratic example of Section \ref{se:LQ} is independent of Sections \ref{se:existence_GMFG_sol}--\ref{se:approxEQ}. Sections \ref{se:dependence-on-U}-\ref{sec:conv_emp_meas} provide preliminary results for the proofs of Section \ref{se:approxEQ}, namely the dependence of the optimal control on the labelling and the convergence of neighborhood empirical measures under various assumptions, respectively. Section \ref{se:approxEQ} is devoted to the proofs of the Theorems of Section \ref{subse:main_res_approx_eq}.

\section{Existence of graphon equilibria}\label{se:existence_GMFG_sol}

This section proves Theorem \ref{th:existence}, by adapting the strategy of \cite{lacker2015mean}.
In particular, we will make use of the notion of \emph{relaxed controls}, developed in Section \ref{se:relaxed}. In this section, we fix a graphon $W \in L^1_+[0,1]^2$. Note that $W$ is not necessarily bounded.

\subsection{Continuity of the $\textsf{W}$ operator} \label{se:W-continuity}

First, we compile some essential continuity properties of the operator $\mathsf{W}$ defined in \eqref{def:Woperator-measure}.  These results will be useful in more general forms, so we work here with a Polish space $E$ which will later be either $E=\R^d$ or the path space $E=\C^d = C([0,T];\R^d)$. Recall that $\Punif([0,1] \times E)$ is the set of probability measures on $[0,1] \times E$ with uniform first marginal, endowed with the topology of weak convergence.

We first recall a well known fact that continuity assumptions for test functions can be relaxed when dealing with weak convergence of joint distributions with a common marginal:

\begin{lemma} \cite[Lemma 2.1]{beiglbock2018denseness} \label{le:Punif-continuity}
Suppose $h : [0,1] \times E \to \R$ is bounded and measurable, with $h(u,\cdot)$ continuous on $E$ for a.e.\ $u \in [0,1]$. Then $\Punif([0,1] \times E) \ni \mu \mapsto \langle \mu, h\rangle$ is continuous.
\end{lemma}

The next lemma is the main result of this section. Part (2) will not be needed but is illustrative and not much longer to prove.

\begin{lemma} \label{le:Wcontinuity}
The following continuity properties hold:
\begin{enumerate}
\item For a.e.\ $u \in [0,1]$, the following map is continuous:
\begin{align*}
\Punif([0,1] \times E) \ni \mu \mapsto \text{\rm\textsf{W}}\mu(u) \in \M_+(E).
\end{align*}
\item Suppose that the map $[0,1] \ni u \mapsto W(u,\cdot) \in (L^1[0,1],\text{weak})$ is continuous. 
 Then, for each $\mu \in \Punif([0,1] \times E)$ and each bounded measurable function $\varphi :E \to \R$, the map $u \mapsto \langle \text{\rm\textsf{W}}\mu(u), \varphi\rangle$ is continuous.
\item Suppose that the map $[0,1] \ni u \mapsto W(u,v)dv \in \M_+([0,1])$ is continuous. 
Suppose $\mu \in \Punif([0,1] \times E)$ is such that there exists a version of the disintegration $u \mapsto \mu_u$ which is continuous. Then the following map is continuous:
\begin{align*}
[0,1] \ni u \mapsto  \text{\rm\textsf{W}}\mu(u) \in \M_+(E).
\end{align*}
\end{enumerate}
\end{lemma}
\begin{proof}
We first prove (1). Since $W \in L^1[0,1]^2$, it holds by Fubini's theorem that $W(u,\cdot) \in L^1[0,1]$ for a.e.\ $u \in [0,1]$. Fix such a $u$ as well as a bounded continuous $\varphi : \R^d \to \R$. Write
\begin{align*}
\langle \textsf{W}\mu(u),\varphi\rangle &= \int_{[0,1] \times \R^d} W(u,v)\varphi(x)\,\mu(dv,dx) = \int_{\R} x\,F_{\mu}(dx),
\end{align*}
where $F_{\mu}$ is the image of $\mu$ under the map $(v,x) \mapsto W(u,v)\varphi(x)$. We first claim that 
\[
\Punif([0,1] \times \R^d) \ni \mu \mapsto F_{\mu} \in \P(\R)
\]
is continuous. To see this, note for bounded continuous $h : \R \to \R$ that
\begin{align*}
\langle F_{\mu},h\rangle = \int_{[0,1] \times \R^d} h(W(u,v)\varphi(x))\,\mu(dv,dx).
\end{align*}
The bounded function $h(W(u,v)\varphi(x))$ depends continuously on $x$ and measurably on $v$, and it follows from Lemma \ref{le:Punif-continuity} that $\mu \mapsto\langle F_{\mu},h\rangle$ is continuous on $\Punif([0,1] \times \R^d)$. To finally deduce that $\int_{\R} x\,F_{\mu}(dx)$ depends continuously on $\mu \in \Punif([0,1] \times \R^d)$, simply note that we have the uniform integrability bound
\begin{align*}
\int_{\R} |x|1_{\{|x| \ge r\}}\,F_{\mu}(dx) &= \int_{[0,1] \times \R^d} |W(u,v)\varphi(x)|1_{\{|W(u,v)\varphi(x)| \ge r\}}\, \mu(dv,dx) \\
	&\le \|\varphi\|_\infty \int_0^1 |W(u,v)|1_{\{|W(u,v)| \ge r/\|\varphi\|_\infty\}}\, dv
\end{align*}
for any $r > 0$, which tends to zero as $r\to\infty$, because $W(u,\cdot) \in L^1[0,1]$.

To prove (2), fix $\mu \in \Punif([0,1] \times E)$. Let $u_n \to u$ in $[0,1]$, and let $\varphi : E \to \R$ be bounded and continuous. Let $\psi(v) = \E[\varphi(X)\,|\,U=v]$, for $(U,X) \sim \mu$. Then $\psi \in L^\infty[0,1]$, and so the weak convergence of $W(u_n,\cdot) \to W(u,\cdot)$ in $L^1[0,1]$ implies
\begin{align*}
\langle \textsf{W}\mu(u_n),\varphi\rangle &= \int_{[0,1] \times E} W(u_n,v) \varphi(x) \, \mu(dv,dx) = \int_0^1 W(u_n,v) \psi(v) \, dv \\
	&\to \int_0^1 W(u,v) \psi(v) \, dv  = \langle \textsf{W}\mu(u),\varphi\rangle .
\end{align*}
The proof of (3) is similar to that of (2), except that we must simply note that $\psi$ is continuous in order to justify the convergence.
\end{proof}

\subsection{The relaxed formulation} \label{se:relaxed}
A relaxed control is a measure on $[0,T] \times A$ with first marginal  equal to the Lebesgue measure. We will denote $\V$ the set of relaxed controls, equipped with the topology of weak convergence, which makes $\V$ a compact (since $A$ is compact) metric space. For each $q \in \V$, we can identify the measurable map $t \mapsto q_t \in \P(A)$ that arises form the disintegration $q(dt, da) = dtq_t(da)$, and which is unique up to (Lebesgue) almost everywhere equality. \emph{Strict controls} are relaxed controls $q \in \V$ of the form $q_t = \delta_{\alpha(t)}$ for a.e.\ $t$, for some measurable $\alpha : [0,T] \to A$.

We will work in this section on the space $\Omega := \V \times [0,1] \times \C^d$. This Polish space is endowed with its Borel $\sigma$-field. In the following, a generic element of $\Omega$ is denoted $(q,u,x)$ and the coordinate maps on $\V$, $[0,1]$, and $\C^d$ are denoted $\Lambda$, $U$, and $X$ respectively.
The canonical filtration $\FF=(\F_t)_{t \in [0,T]}$ is defined by letting $\F_t$ denote the $\sigma$-field generated by $\Lambda|_{[0,t] \times A}, U$, and $(X_s)_{s \in [0,t]}$.

Let $C^{\infty}_c(\R^d)$ denote the set of infinitely differentiable functions $\varphi : \R^d \to \R$ with compact support, and let $\nabla\varphi$ and $\nabla^2\varphi$ denote respectively the gradient and the Hessian of $\varphi$. Define the generator $L$ on $\varphi \in C^{\infty}_c(\R^d)$ by
\begin{equation*}
L\varphi(t,x,a) := b(t,x,a) \cdot \nabla \varphi(x) + \frac{1}{2}\tr{\left(\sigma \sigma^\top(t,x) \nabla^2 \varphi(x)\right)},
\end{equation*}
for $(t,x,a) \in [0,T] \times \R^d \times A$.
For $\varphi \in C^{\infty}_c(\R^d)$ we define a process $N_t^{\varphi} : \Omega \to \R$ by
\begin{equation*}
N_t^{\varphi}(q,u,x) := \varphi(x_t) - \int_{[0,t] \times A} L \varphi(s,x_s,a)\,a(ds,da), \qquad t \in [0,T].
\end{equation*}
The set of admissible laws $\RC$ is defined as the set of $P \in \P(\Omega)$ satisfying
\begin{enumerate} 
\item $P \circ (U,X_0)^{-1} = \lambda$.
\item For each $\varphi \in C^{\infty}_c(\R^d)$, the process $(N_t^{\varphi})_{t \in [0,T]}$ is a $P$-martingale.
\end{enumerate}
This gives a martingale problem formulation, in the spirit of Stroock-Varadhan \cite{stroock1997multidimensional}, for the controlled state processes in \eqref{def:SDEgraphongame}.

For $\mu \in \Punif([0,1] \times \C^d)$ representing the fixed population distribution, we write $\mu_t$ for the marginal obtained as the image by $(u,x) \mapsto (u,x_t)$, and we define a random variable $\Gamma^{\mu} : \Omega \to \R$ by
\begin{equation}
\Gamma^{\mu}(q,u,x) := g(x_T, \textsf{W}\mu_T(u)) + \int_{[0,T] \times A} f(t, x_t, \textsf{W}\mu_t(u), a) \,q(dt,da). \label{def:Gamma}
\end{equation}

\begin{remark} \label{re:relaxation}
Recalling the notation of Section \ref{se:graphongame}, if $\alpha \in \Aunif$, then $dt\delta_{\alpha(t,U,X^\alpha_t)}(da)$ is a random element of $\V$, and the joint law $P^\alpha$ of $(dt\delta_{\alpha(t,U,X^\alpha_t)}(da), U,X^\alpha)$ defines an element of $\RC$. Indeed, the condition $(U,X^\alpha_0) \sim \lambda$ was imposed in Section \ref{se:graphongame}, and the defining martingale property (2) of $\RC$ follows immediately from It\^o's formula. Unpacking the notation, it holds also that 
\begin{align}
J_W(\mu,P^\alpha) = \langle P^\alpha,\Gamma^{\mu}\rangle. \label{eq:relaxation-JW}
\end{align}
\end{remark}

Given $\mu \in \Punif([0,1] \times \C^d)$, a single player's objective is to find
\begin{align*}
\RC^*(\mu) &:= \mbox{arg} \underset{P \in \RC}{\max} \ \langle P, \Gamma^{\mu}\rangle := \{ P \in \RC : \langle P, \Gamma^{\mu}\rangle \ge \langle Q, \Gamma^{\mu}\rangle \ \forall Q \in \RC\}.
\end{align*}
Our first goal will be to prove the existence of what one might naturally call a \emph{relaxed $W$-equilibrium}, defined as a fixed point of the set-valued  map $\Phi : \Punif([0,1] \times \C^d) \to 2^{\Punif([0,1] \times \C^d)}$ given by
\begin{align*}
\Phi(\mu) := \{ P \circ (U,X)^{-1} : P \in \RC^*(\mu)\}.
\end{align*}
That is, a \emph{relaxed $W$-equilibrium} is any $\mu \in \Punif([0,1] \times \C^d)$ satisfying $\mu \in \Phi(\mu)$. We will first prove the existence of such a fixed point in Proposition \ref{pr:relaxedexistence}, and then we will show how to turn it into a true $W$-equilibrium in the sense of Section \ref{se:graphongame}.

\subsection{Existence of relaxed equilibrium}

The goal of this section is to prove the following:

\begin{proposition} \label{pr:relaxedexistence}
There exists $\mu \in \Punif([0,1] \times \C^d)$ such that $\mu \in \Phi(\mu)$.
\end{proposition}

To do so, we will use the following lemma on continuity:

\begin{lemma} \label{le:continuity-cost}
The following map is jointly continuous:
\[
\Punif([0,1] \times \C^d) \times  \RC \ni (\mu,P) \mapsto \langle P,\Gamma^{\mu} \rangle \in \R
\]
\end{lemma}
\begin{proof}
Lemma \ref{le:Wcontinuity}(1) and the boundedness and continuity of $g$ and $f$ imply that $\Gamma^{\mu}(q,u,x)$ is a continuous function of $(q,x,\mu)$ for a.e.\ $u \in [0,1]$; see \cite[Corollary A.5]{lacker2015mean}. In addition, $\Gamma^{\mu}$ is measurable on $\Omega$. The claim follows by applying Lemma \ref{le:Punif-continuity} with $E=\V \times \C^d$, noting that $\RC$ can be viewed as a subset of $\Punif([0,1] \times E)$.
\end{proof}

\begin{proof}[Proof of Proposition \ref{pr:relaxedexistence}]
We will apply the Kakutani-Fan-Glicksberg fixed point theorem \cite[Theorem 1]{fan1952fixed}, which requires that we identify a nonempty compact convex set $K \subset \P([0,1] \times \C^d)$ such that:
\begin{enumerate}
\item $\Phi(\mu) \subset K$ for each $\mu \in K$.
\item $\Phi( \mu )$ is nonempty and convex for each $\mu \in K$.
\item The graph $\{(\mu ,\mu') : \mu \in K, \ \mu'  \in \Phi(\mu )\}$ is closed.
\end{enumerate}
A good choice turns out to be $K := \{P \circ (U,X)^{-1} : P \in \RC\}$.
Property (1) is then clearly satisfied, because $\RC^*(\mu) \subset \RC$ for all $\mu$.

Let us prove that $K$ is compact and convex, as it is easily seen to be nonempty. First, note that $\RC$ is the set of $P \in \P(\Omega)$ satisfying $P \circ (U,X_0)^{-1} = \lambda$ and
\begin{align}
\langle P, h(N^\varphi_t-N^\varphi_s)\rangle = 0, \label{eq:mtgconstraint}
\end{align}
for all $T \ge t > s \ge 0$, $\varphi \in C^\infty_c(\R^d)$, and bounded continuous $\F_s$-measurable functions $h$ (which generate the $\sigma$-field $\F_s$).
This shows clearly that $\RC$ is convex, and thus so is $K$. To see that $\RC$ is closed, note that the continuity of $(b,\sigma)$ ensure that $L\varphi$ is jointly continuous, and thus so is $N^\varphi_t : \Omega \to \R$ by \cite[Corollary A.5]{lacker2015mean}. It follows that \eqref{eq:mtgconstraint} is closed under weak limits, and so $\RC$ is a closed set. To see that $\RC$ is pre-compact, we note easily that it is tight, because $[0,1] \times \V$ is compact, and because $\{P \circ X^{-1} : P \in \RC\}$ is easily seen to be tight as a consequence of the boundedness of $(b,\sigma)$, e.g., by \cite[Theorem 1.4.6]{stroock1997multidimensional}.
The compactness of $K$ follows from compactness of $\RC$, because the map $P \mapsto P \circ (U,X)^{-1}$ is continuous.

Next, for each $\mu \in K$, note that  $\RC^*(\mu)$ is nonempty as a consequence of the continuity of $P \mapsto \langle P,\Gamma^{\mu}\rangle$ from Lemma \ref{le:continuity-cost} and compactness of $\RC$ shown above; it follows that $\Phi(\mu)$ is also nonempty. Convexity of $\RC^*(\mu)$ follows from the linearity of $P \mapsto \langle P,\Gamma^{\mu}\rangle$ and the convexity of $\RC$. In turn, convexity of $\Phi(\mu)$ follows from convexity of $\RC^*(\mu)$ and linearity of the map $P \mapsto P \circ (U,X)^{-1}$.

It remains to prove the closedness of the graph of $\Phi$ as in (3) above. By continuity of $P \mapsto P\circ (U,X)^{-1}$ and compactness of $K$, it suffices to prove the closedness of
\[
\{(\mu,P) : \mu \in K, \, P \in \RC^*(\mu)\}.
\]
Suppose $\mu_n \to \mu$ and $P_n \to P$, with $\mu,\mu_n \in K$, $P_n \in \RC^*(\mu_n)$, and $P \in \RC$. To show that $P \in \RC^*(\mu)$, we must show that $\langle P,\Gamma^{\mu}\rangle \ge \langle Q,\Gamma^{\mu}\rangle$ for every $Q \in \RC$.
This follows easily from the joint continuity of  Lemma \ref{le:continuity-cost}(2), which yields
\begin{align*}
\langle P,\Gamma^{\mu}\rangle &= \lim_n \langle P_n,\Gamma^{\mu_n}\rangle \ge \lim_n \langle Q,\Gamma^{\mu}_n\rangle = \langle Q,\Gamma^{\mu}\rangle,
\end{align*}
with the inequality coming from the assumption $P_n \in \RC^*(\mu_n)$. This completes the proof.
\end{proof}

\subsection{Construction of Markovian equilibrium} \label{se:Markov}

We now construct a Markovian equilibrium, as defined in Section \ref{se:graphongame}, thereby proving Theorem \ref{th:existence}. We follow the strategy of the proof of \cite[Theorem 3.7]{lacker2015mean}, based on Markovian projection  \cite{brunick2013mimicking}.
This section makes heavier use of the notation $(\Lambda,U,X)$ for the coordinate maps on $\Omega$.

Let $\mu$ be any fixed point,  $\mu \in \Phi(\mu)$, the existence of which is guaranteed by Proposition \ref{pr:relaxedexistence}. Note that $\mu \in \Phi(\mu)$ is equivalent to the existence of $P \in \RC^*(\mu)$ such that $\mu = P \circ (U,X)^{-1}$. Since $P \in \RC$, the definition of $\RC$ and a standard martingale problem argument (e.g., \cite[Theorem 2.5]{nicole1987compactification}) shows that there exists a $P$-Brownian motion $B$ such that 
\begin{equation*}
dX_t = \int_A b(t, X_t, a) \Lambda_t(da)dt + \sigma(t, X_t) dB_t.
\end{equation*}

To handle the additional variable $U$, we simply note that the $(d+1)$-dimensional process $(U, X_t)_{t \in [0,T]}$ is an It\^o process in its own right, 
\begin{align*}
d\begin{pmatrix}
U \\ X_t
\end{pmatrix} = \begin{pmatrix}
0 \\\int_A b(t, X_t, a) \Lambda_t(da) 
\end{pmatrix}\,dt + \begin{pmatrix}
 0_{d_0}^\top \\ \sigma(t, X_t) 
\end{pmatrix} d B_t.
\end{align*}

Consider jointly measurable functions $(\widehat{b},\widehat{f}) : [0,T] \times [0,1] \times \R^d \to \R^d \times \R$ satisfying
\begin{align*}
\widehat{b}(t,U,X_t) &= \E\bigg[ \int_A b(t, X_t, a)  \Lambda_t(da) \,\bigg|\, U,X_t\bigg], \\
\widehat{f}(t,U,X_t) &= \E\bigg[ \int_A f(t, X_t,\textsf{W}\mu_t(U), a)  \Lambda_t(da) \,\bigg|\, U,X_t\bigg], \ \ \ P-a.s., \ a.e. \ t \in [0,T].
\end{align*}
Such functions exist by \cite[Proposition 5.1]{brunick2013mimicking}.
Applying the mimicking theorem \cite[Corollary 3.7]{brunick2013mimicking}, we may find a process $(\widehat{U}_t,\widehat{X}_t)_{t \in [0,T]}$, perhaps on another probability space $(\widehat\Omega,\widehat\F,\widehat\PP)$ with another Brownian motion $\widehat{B}$, solving the SDE
\begin{align}
d\begin{pmatrix}
 \widehat{U}_t \\ \widehat{X}_t
\end{pmatrix} = \begin{pmatrix}
0 \\ \widehat{b}(t,U,X_t) 
\end{pmatrix}\,dt + \begin{pmatrix}
0_{d_0}^\top \\ \sigma(t, \widehat{X}_t) 
\end{pmatrix} d \widehat{B}_t, \label{pf:Xhat-SDE}
\end{align}
and satisfying $(\widehat{U}_t,\widehat{X}_t) \stackrel{d}{=} (U,X_t)$ for each $t \in [0,T]$. Part of the definition of an SDE solution, of course, is that $\widehat{B}$ is a Brownian motion relative to the filtration $\widehat{\FF}=(\widehat{\F}_t)_{t \in [0,T]}$ generated by $(\widehat{U},\widehat{X},\widehat{B})$. From the dynamics \eqref{pf:Xhat-SDE} we deduce that $\widehat{U}_t=\widehat{U}_0$ for all $t$, which implies that $\widehat{U} := \widehat{U}_0$ is Unif$[0,1]$ since $U$ is. Hence, $\widehat{U}$ is a.s.\ $\widehat{\F}_0$-measurable, and in particular independent of $B$.

Now, for $(t,x,m) \in [0,T] \times \R^d \times \M_+(\R^d)$ let  $\mathcal{S}(t,x,m) \subset \R^d \times \R$ denote the set defined in \eqref{asmp:convex}. From its assumed convexity, we deduce that $(\widehat{b}(t,U,X_t),\widehat{f}(t,U,X_t))$ belongs a.s.\ to $\mathcal{S}(t,X_t,\textsf{W}\mu_t(U))$. Thus, using a measurable selection result from \cite[Theorem A.9]{haussmann1990existence}, there exist measurable functions $\widehat{\alpha} : [0,T] \times [0,1] \times \R^d \to A$ and $\widehat{z} : [0,T] \times [0,1] \times \R^d \to \R_+$ such that,

$P$-a.s., for a.e.\ $t \in [0,T]$,
\begin{align}
\widehat{b}(t,U,X_t) &= b(t, X_t, \widehat{\alpha}(t,U,X_t)), \label{eq:bhat} \\
\widehat{f}(t,U,X_t) &= f(t, X_t,\textsf{W}\mu_t(U), \widehat{\alpha}(t,U,X_t)) - \widehat{z}(t,U,X_t). \label{eq:fhat}
\end{align}
Applying \eqref{eq:bhat}, the dynamics \eqref{pf:Xhat-SDE} can then be written as
\begin{align*}
d\widehat{X}_t = b(t, \widehat{X}_t, \widehat{\alpha}(t,\widehat{U},\widehat{X}_t))dt + \sigma(t,\widehat{X}_t)d\widehat{B}_t.
\end{align*}
Note that $\widehat\alpha$ belongs to $\Aunif$, as defined in Section \ref{se:graphongame}.
By uniqueness of the SDE, in the notation of Section \ref{se:graphongame}, we have
\begin{align}
(\widehat{U},\widehat{X}) \stackrel{d}{=} (U,X^{\widehat{\alpha}}).
\end{align}
As in Remark \ref{re:relaxation}, the joint law $P^{\widehat\alpha}$ of $(dt\delta_{\widehat\alpha(t,\widehat{U},\widehat{X}_t)}(da),\widehat{U}, \widehat{X})$ is thus an element of $\RC$. Let $\widehat\mu \in \Punif([0,1] \times \C^d)$ denote the joint law of $(\widehat{U},\widehat{X})$. 
Then, as in \eqref{eq:relaxation-JW}, we have
\begin{align} 
\langle P^{\widehat\alpha},\Gamma^{\widehat\mu}\rangle &=  J_W(\widehat\mu,\widehat\alpha). \label{pf:optimal-value-Phat}
\end{align}
We will complete the proof by showing that in fact  $\langle P^{\widehat\alpha},\Gamma^{\widehat\mu}\rangle \ge J_W(\widehat\mu,\alpha)$ for all $\alpha \in \Aunif$. Again using \eqref{eq:relaxation-JW}, it suffices to show 
 that $P^{\widehat\alpha} \in \RC^*(\widehat\mu)$, i.e., $\langle P^{\widehat\alpha},\Gamma^{\widehat\mu}\rangle \ge \langle Q,\Gamma^{\widehat\mu}\rangle$ for all $Q \in \RC$.

To this end, note that $(\widehat{U},\widehat{X}_t) \stackrel{d}{=}(U,X_t)$ and thus $\widehat\mu_t=\mu_t$ for each $t \in [0,T]$. Since $\Gamma^\mu$, defined in \eqref{def:Gamma}, depends on $\mu$ only through its marginals $(\mu_t)_{t \in [0,T]}$, 
\begin{align}
\Gamma^{\mu}(q,u,x)=\Gamma^{\widehat\mu}(q,u,x), \quad \text{ for all } \ (q,u,x) \in \Omega. \label{pf:Gamma-bar-vs-hat}
\end{align}
Hence,
\begin{align*}
\langle P^{\widehat\alpha},\Gamma^{\widehat\mu}\rangle &= \langle P^{\widehat\alpha},\Gamma^{\mu}\rangle = \widehat\E \left[ \int_0^T f(t,\widehat{X}_t,\textsf{W}\mu_t(\widehat{U}),\widehat\alpha(t,\widehat{U},\widehat{X}_t))dt + g(\widehat{X}_T,\textsf{W}\mu_T(\widehat{U}))\right],
\end{align*}
where $\widehat\E$ denotes expectation on $(\widehat\Omega,\widehat\F,\widehat\PP)$. Using Fubini's theorem and the equality in law $(\widehat{U},\widehat{X}_t) \stackrel{d}{=}(U,X_t)$ for each $t \in [0,T]$, we find
\begin{align*}
\langle P^{\widehat\alpha},\Gamma^{\widehat\mu}\rangle &=  \E\left[ \int_0^T f(t,X_t,\textsf{W}\mu_t(U),\widehat\alpha(t,U,X_t))dt + g(X_T,\textsf{W}\mu_T(U))\right].
\end{align*}
The identity \eqref{eq:fhat} and the definition of $\widehat{f}$ imply
\begin{align*}
f(t, X_t,\textsf{W}\mu_t(U), \widehat{\alpha}(t,U,X_t)) &\ge \E\bigg[ \int_A f(t, X_t,\textsf{W}\mu_t(U), a)  \Lambda_t(da) \,\bigg|\, U,X_t\bigg].
\end{align*}
Using this, the tower property, and the definition of $P$, we deduce
\begin{align*}
\langle P^{\widehat\alpha},\Gamma^{\widehat\mu}\rangle &\ge \E\left[ \int_0^T \int_A f(t, X_t,\textsf{W}\mu_t(U), a)  \Lambda_t(da) dt + g(X_T,\textsf{W}\mu_T(U))\right] = \langle P,\Gamma^{\mu}\rangle.
\end{align*}
We know by assumption that $P \in \RC^*(\mu)$. Hence, for any $Q \in \RC$, we have $\langle P,\Gamma^{\mu}\rangle \ge \langle Q,\Gamma^{\mu}\rangle$. Using again \eqref{pf:Gamma-bar-vs-hat}, we deduce finally that $\langle P^{\widehat\alpha},\Gamma^{\widehat\mu}\rangle \ge \langle Q,\Gamma^{\mu}\rangle=\langle Q,\Gamma^{\widehat\mu}\rangle$ for all $Q \in \RC$, which completes the proof of Theorem \ref{th:existence} as explained above. \hfill \qedsymbol

\begin{remark} \label{re:Markov-optimality}
The argument in this section shows that, for each $\mu$,
\begin{align}
\sup_{\beta \in \Aunif}J_W(\mu,\beta) &= \sup_{P \in \RC}\langle P, \Gamma^{\mu}\rangle. \label{pf:DPP-value1}
\end{align}
That is, Markovian controls achieve the same value as the more general controls allowed in $\RC$, which may depend on additional randomness. See \cite{haussmann1990existence,nicole1987compactification}  for more general studies of this well known principle.
\end{remark}

\subsection{The case of constant degree}

\begin{proof}[Proof of Proposition \ref{prop:constdegree}]
With $\nu \in C([0,T];, \P(\R^d)$ and $\mu_t =\mathrm{Unif}[0,1] \times \nu_t$ as in the statement of the proposition, the key point is the simple identity $\nu_t=\textsf{W}\mu_t(u)$. Indeed, for bounded measurable $\varphi : \R^d \to \R$, we have
\begin{align*}
\langle \textsf{W}\mu_t(u),\varphi\rangle &= \int_{[0,1]\times \R^d} W(u,v)\, \varphi(x) \,\mu_t(dv, dx) \\
	&=\int_0^1 \int_{\R^d} W(u,v)\, \varphi(x) \,\nu_t(dx)\,dv \\
	&= \int_{\R^d} \varphi(x) \,\nu_t(dx),
\end{align*}
with the last identity following from Fubini's theorem and the assumption \eqref{def:const-degree}. Then $J_{1}(\nu,\alpha) = J_W(\mu,\alpha)$ for any $\alpha\in \A_1$. Since $\nu$ is a mean field equilibrium with control $\alpha^*$, we have
\begin{align*}
J_W(\mu,\alpha^*) = J_{1}(\nu,\alpha^*) = \sup_{\alpha\in \A_1}J_{1}(\nu,\alpha) = \sup_{\alpha\in \A_1}J_W(\mu,\alpha).
\end{align*}
The only remaining subtlety is to argue that $\sup_{\alpha\in \A_1}J_W(\mu,\alpha)=\sup_{\alpha\in \Aunif}J_W(\mu,\alpha)$. That is, the optimal value is the same regardless of whether one allows the controls to depend on an independent uniform $U$. This can be argued by way of a Markovian projection argument as in Section \ref{se:Markov}, or by directly applying \cite[Theorem 3.7]{lacker2015mean}.
\end{proof}

\subsection{Uniqueness } \label{subse:proof_unique}

This section proves Proposition \ref{prop:unique_W_eq}, relying on the recasting of the graphon game as a mean field game as  in Section \ref{se:PDE}. The key point is that the monotonicity condition (3) in Proposition \ref{prop:unique_W_eq} translates precisely to the usual Lasry-Lions monotonicity condition for the associated mean field game. Using the same notation of Section \ref{se:PDE}, (3) implies
\begin{align*}
\int_{[0,1] \times \R^d} &(\overline{g}(\overline{x}, {m}_1) - \overline{g}(\overline{x}, {m}_2)({m}_1 - {m}_2)(du,dx) \\
    	&= \int_{[0,1] \times \R^d} (g(x, \textsf{W}{m}_1(u)) - g(x, \textsf{W}{m}_2(u))({m}_1 - {m}_2)(du,dx) \le 0,
\end{align*}
for $m_1,m_2  \in \Punif([0,1] \times \R^d)$.
Similarly, $\overline{f}$ takes the form
 \begin{align*}
   \overline{f}(t, \overline{x}, {m},a) = f_1(t, x, a) + \overline{f}_2(t, \overline{x}, {m}), \ \text{ where }  \ \overline{f}_2(t, \overline{x}, {m}) = f_2(t, x, \textsf{W}{m}(u)),
 \end{align*}
and for all ${m}_1, {m}_2$ we have 
\begin{align*}
    \int_{[0,1] \times \R^d} &(\overline{f}_2(t,\overline{x}, {m}_1) - \overline{f}_2(t,\overline{x}, {m}_2)({m}_1 - {m}_2)(du,dx) \le 0.
\end{align*}
This shows that the mean field game of Section \ref{se:PDE} satisfies the Lasry-Lions monotonicity condition. The classical uniqueness proof from mean field game theory then applies; see \cite[Theorem 8.10]{lacker2018mean} for a short proof which applies directly in our context.

\begin{remark} \label{re:monotonicity-examples}
We mention here two classes of examples of $g$ satisfying \eqref{def:monotonicity}:
\begin{enumerate}
\item Suppose $\varphi :\R^d \to \R$ is bounded and continuous, and let
\begin{align*}
g(x,m) = \bigg(\varphi(x) - \int_{\R^d}\varphi(y)\,m(dy)\bigg)^2, \quad x \in \R^d, \, m \in \M_+(\R^d),
\end{align*}
A straightforward calculation shows that the left-hand side of \eqref{def:monotonicity} equals
\begin{align*}
-\int_0^1 \int_0^1 W(u,v)\psi(u)\psi(v)\,du\,dv, \quad\text{where } \ \psi(u) := \int_{\R^d}\varphi(x)\,({m}^1_u-{m}^2_u)(dx),
\end{align*}
for $m_i(du,dx)=dum^i_u(dx) \in \Punif([0,1] \times \R^d)$, $i=1,2$. Hence, if $W$ is positive semidefinite, we obtain the monotonicity condition \eqref{def:monotonicity}.
\item Suppose $g(x,m)=\int_{\R^d}\varphi(x,y)m(dy)$, where $\varphi :(\R^d)^2 \to \R$ is bounded and continuous. Then the left-hand side of \eqref{def:monotonicity} equals
\begin{align*}
\int_{[0,1]\times\R^d}\int_{[0,1]\times\R^d}W(u,v)\varphi(x,y)({m}_1 - {m}_2)(dv,dy)({m}_1 - {m}_2)(du,dx).
\end{align*}
This is nonpositive if, for instance, $W$ is positive semidefinite and $\varphi$ is negative semidefinite when viewed as integral operators, so that the tensor product of these two operators is negative definite.
\end{enumerate}
\end{remark}

\section{On the dependence of optimal controls on $U$} \label{se:dependence-on-U}

This short section develops two lemmas which will be used solely in the proof of Theorem \ref{th:approxEQ-continuous}, in Section \ref{se:pf:cont-kernel}. We give these results here because the proofs use the same relaxed formulation of Section \ref{se:existence&uniqueness}, particularly the Markovian projection of Section \ref{se:Markov}.

For this section, we fix $W \in L^1_+[0,1]^2$ and $\mu_\cdot \in C([0,T];\Punif([0,1] \times \R^d))$, and we introduce the following notation.
For $u \in [0,1]$, $m \in \P(\R^d)$, and $\alpha \in \A_1$, let $X^{m,\alpha}$ denote the unique in law solution of the SDE
\begin{align*}
dX^{m,\alpha}_t = b(t,X^{m,\alpha}_t,\alpha(t,X^{m,\alpha}_t))dt + \sigma(t,X^{m,\alpha}_t)dB_t, \qquad X^{m,\alpha}_0 \sim m,
\end{align*}
and define
\begin{align*}
J_W^{u,m}(\mu_\cdot,\alpha) := \E \left[\int_0^T f(t, X^{m,\alpha}_t, \text{\rm\textsf{W}}\mu_t(u), \alpha(t,X^{m,\alpha}_t)) dt + g(X^{m,\alpha}_T, \text{\rm\textsf{W}}\mu_T(u)) \right].
\end{align*}

The first lemma states essentially that, if $\alpha^* \in \Aunif$ is optimal for the given $\mu_\cdot$, then the control $(t,x) \mapsto \alpha^*(t, u , x)$ is  still optimal if we freeze the ``label" variable $U=u$, for almost every $u$. Recall in the following that $\lambda \in \Punif([0,1] \times \R^d)$ denotes the initial law, and $\lambda(du,dx)=du\lambda_u(dx)$ its disintegration.

\begin{lemma}\label{le:optimality_alpha_u}
Suppose $\alpha \in \Aunif$ satisfies $J_W(\mu_\cdot,\alpha) \ge J_W(\mu_\cdot,\beta)$ for all $\beta \in \Aunif$.
Then 
\begin{align*}
J_W^{u,\lambda_u}(\mu_\cdot,\alpha_u)  = \sup_{\beta \in \A_1} J_W^{u,\lambda_u}(\mu_\cdot,\beta),  \ \ \text{ for a.e. } u \in [0,1],
\end{align*}
where we define $\alpha_u\in\A_1$ by $\alpha_u(t,x) := \alpha(t,u,x)$. 
\end{lemma}
\begin{proof}
Recall the identity \eqref{pf:DPP-value1} from Remark \ref{re:Markov-optimality}.
For $u \in [0,1]$ and $m \in \P(\R^d)$, let us define $\RC_{u,m}$ as the set of $P \in \P(\Omega)$ such that $P \circ (U,X_0)^{-1} = \delta_u \times m$ and such that $(N^\varphi_t)_{t \in [0,T]}$ is a $P$-martingale for each $\varphi \in C^{\infty}_c(\R^d)$. The same argument as in Section \ref{se:Markov} which led to \eqref{pf:DPP-value1} (see also \cite[Theorem 3.7]{lacker2015mean} or \cite[Corollary 6.8]{nicole1987compactification}) shows that
\begin{align}
\sup_{\beta \in \A_1}J^{u,m}_W(\mu_\cdot,\beta) &= \sup_{P \in \RC_{u,m}}\langle P, \Gamma^{\mu_\cdot}\rangle. \label{pf:DPP-value2}
\end{align}
It is straightforward to check that $\{(u,P) : u \in [0,1], \ P \in \RC_{u,\lambda_u}\}$ is a Borel set in $[0,1] \times \P(\Omega)$.
Since the map $\P(\Omega) \ni P \mapsto \langle P, \Gamma^{\mu_\cdot}\rangle$ is Borel, a standard measurable selection theorem \cite[Proposition 7.50]{bertsekas1996stochastic} then shows that $u \mapsto \sup_{P \in \RC_{u,\lambda_u}}\langle P, \Gamma^{\mu_\cdot}\rangle$ is universally measurable, and
\begin{align*}
\int_0^1 \sup_{P \in \RC_{u,\lambda_u}}\langle P, \Gamma^{\mu_\cdot}\rangle \, du = \sup\bigg\{ \int_0^1 \langle P_u, \Gamma^{\mu_\cdot}\rangle\,du : P_\cdot \text{ Borel, } P_u \in \RC_{u,\lambda_u} \ a.e. \ u \bigg\},
\end{align*}
where ``$P_\cdot$ Borel" means that the map $[0,1] \ni u \mapsto P_u \in \P(\Omega)$ is Borel measurable. From the definitions, and noting that $\int_0^1 \delta_u \times \lambda_u \,du = \lambda$, it is straightforward to check that $\int_0^1 P_u\,du$ belongs to $\RC$ whenever $P_u \in \RC_{u,\lambda_u}$ for a.e.\ $u$. Conversely, if $P \in \RC$, then the regular conditional measure $P_u := P(\cdot\,|\,U=u)$ belongs to $\RC_{u,\lambda_u}$ for a.e.\ $u$. It follows that 
\begin{align*}
\sup\bigg\{ \int_0^1 \langle P_u, \Gamma^{\mu_\cdot}\rangle\,du : P_\cdot \text{ Borel, } P_u \in \RC_{u,\lambda_u} \ a.e. \ u \bigg\} = \sup_{P \in \RC}\langle P, \Gamma^{\mu_\cdot}\rangle,
\end{align*}
and also that, for $P^* \in \RC$,
\begin{align}
P^* \in \arg\max_{P \in \RC}\langle P, \Gamma^{\mu_\cdot}\rangle \quad \Longleftrightarrow \quad a.e. \ u, \ P^*(\cdot\,|\,U=u) \in \arg\max_{P \in \RC_{u,\lambda_u}}\langle P, \Gamma^{\mu_\cdot}\rangle. \label{pf:selectors1}
\end{align} 
Now, let $\alpha \in \arg\max_{\beta \in \Aunif}J_W(\mu_\cdot,\beta)$ be the given optimizer. In light of \eqref{pf:DPP-value1}, the measure $P^\alpha$ given as in Remark \ref{re:relaxation} then satisfies $P^\alpha \in \arg\max_{P \in \RC}\langle P, \Gamma^{\mu_\cdot}\rangle$. Hence, by \eqref{pf:selectors1}, the conditional measure $P^\alpha(\cdot\,|\,U=u)$ belongs to $\arg\max_{P \in \RC_{u,\lambda_u}}\langle P, \Gamma^{\mu_\cdot}\rangle$ for a.e.\ $u$. But, by well-posedness of the SDEs, we have $P^\alpha(\cdot\,|\,U=u) = P^{\alpha_u}$ for a.e.\ $u$ (cf.\ \cite[Appendix A]{lacker2020convergence}). 
Using \eqref{pf:DPP-value2}, we deduce that $\alpha_u \in \arg\max_{\beta \in \A_1}J^{u,\lambda_u}_W(\mu_\cdot,\beta)$ for a.e.\ $u$, as claimed.
\end{proof}

The next lemma and its corollary justify the claim in Theorem \ref{th:approxEQ-continuous} that assumptions (1) and (2a--d) imply (2). The lemma is a variation on known arguments, such as \cite[Section 5.6]{lacker2020convergence}. Essentially, by working with the relaxed formulation, the set-valued map of optimal control laws can be shown to have closed graph, and the idea is to argue that in certain cases this set-valued map is singleton-valued and thus necessarily continuous.

\begin{lemma} \label{le:continuity-alpha_u}
Suppose conditions (1) and (2a--d) of Theorem \ref{th:approxEQ-continuous} hold.
Then, for each $u \in [0,1]$, there exists a unique optimizer $\alpha^*_u$ for $\sup_{\alpha \in \A_1}J^{u,\lambda_u}_W(\mu_\cdot,\alpha)$. Moreover, the law $\L(X^{\lambda_u,\alpha^*_u})$ depends continuously on $u$.
\end{lemma}
\begin{proof}
We have $b(t,x,a)=b_0(t,x)a+b_1(t,x)$ by assumption.
Fix $u \in [0,1]$ and $m \in \P(\R^d)$. Recall the formula \eqref{pf:DPP-value2} from the proof of Lemma \ref{le:optimality_alpha_u}. We first claim that any optimizer $P$ on the right-hand side of \eqref{pf:DPP-value2} is necessarily of the form $P=\L(dt\delta_{\alpha(t,X^{m,\alpha}_t)}(da),u,X^{m,\alpha})$ for some $\alpha \in \A_1$. To see this, note that we can write $P=\L(\Lambda,u,X )$, where $X$ solves
\begin{align}
dX_t = \Big(b_0(t,X_t)\int_A a\Lambda_t(da) + b_1(t,X_t)\Big)dt + \sigma dB_t, \quad X_0 \sim m. \label{pf:cont-alpha-1}
\end{align}
Letting $\widehat\alpha(t,X_t) = \E\big[\int_A a\Lambda_t(da)\,\big|\,X_t\big]$, and applying the Markovian projection \cite[Corollary 3.7]{brunick2013mimicking},  we find that $X_t\stackrel{d}{=} \widehat{X}_t$ for all $t \in [0,T]$, where $\widehat{X}$ solves the SDE
\begin{align*}
d\widehat{X}_t = \Big(b_0(t,\widehat{X}_t)\widehat\alpha(t,\widehat{X}_t) + b_1(t,\widehat{X}_t)\Big)dt + \sigma dB_t, \quad \widehat{X}_0 \sim m.
\end{align*}
By Jensen's inequality and strict concavity of $f(t,x,m,a)$ in $a$, we have
\begin{align*}
\langle P, \Gamma^{\mu_\cdot}\rangle &= \E \left[\int_0^T \int_A f(t, X_t, \textsf{W}\mu_t(u), a) \Lambda_t(da) dt + g(X_T, \textsf{W}\mu_T(u)) \right] \\
	&\le \E \left[\int_0^T f(t, X_t, \textsf{W}\mu_t(u), \widehat\alpha(t,X_t))  dt + g(X_T, \textsf{W}\mu_T(u)) \right] \\
	&= \E \left[\int_0^T f(t, \widehat{X}_t, \textsf{W}\mu_t(u), \widehat\alpha(t,\widehat{X}_t))  dt + g(\widehat{X}_T, \textsf{W}\mu_T(u)) \right] \\
	&= J_W^{u,m}(\mu_\cdot,\widehat\alpha),
\end{align*}
and this equality is strict unless $\int_A a\Lambda_t(da)=\widehat\alpha(t,X_t)$ a.s.\ a.e. This proves the first claim.

We next claim that in fact there is a unique optimizer $P$ on the right-hand side of \eqref{pf:DPP-value2}. Since we know the optimizers are Markovian, it suffices to show the optimal control $\alpha \in \A_1$ on the left-hand side of \eqref{pf:DPP-value2} is unique up to Lebesgue-a.e.\ equality. To see this, let $\alpha_0,\alpha_1 \in \A_1$ be optimizers. Then $X^i=X^{m,\alpha_i}$ solves
\begin{align*}
dX^i_t = \Big(b_0(t,X^i_t)\alpha_i(t,X^i_t) + b_1(t,X^i_t)\Big)dt + \sigma dB^i_t, \quad X^i_0 \sim m.
\end{align*}
We may assume $X^0$ and $X^1$ are defined on the same probability space, with $(X^0,B^0)$ independent of $(X^1,B^1)$. Let $S$ be a Bernoulli($1/2$) random variable, independent of everything else. Then $X^S$ solves the SDE
\begin{align*}
dX^S_t = \Big(b_0(t,X^S_t)\alpha_S(t,X^S_t) + b_1(t,X^S_t)\Big)dt + \sigma dB^S_t, \quad X^S_0 \sim m,
\end{align*}
where we note that $B^S$ is a Brownian motion.
Define $\widehat\alpha(t,X^S_t)=\E[\alpha_S(t,X^S_t) \,|\,X^S_t]$. Arguing as above via Jensen, we must have $\widehat\alpha(t,X^S_t)= \alpha_S(t,X^S_t)$ a.s.\ a.e. , as otherwise this control would produce a strictly higher reward than $\alpha_0$ or $\alpha_1$.
This implies $\widehat\alpha(t,X^0_t)= \alpha_0(t,X^0_t)$ and $\widehat\alpha(t,X^1_t)= \alpha_1(t,X^1_t)$ a.s.\ a.e. 
The laws of $X^0_t$ and $X^1_t$ have full support for each $t > 0$ by Girsanov's theorem, and we deduce that $\alpha_0=\alpha_1$ a.e.

Finally, knowing that the optimizer $P^*_{u,m} \in \RC_{u,m}$ on the right-hand side of \eqref{pf:DPP-value2} is unique, we will prove that $(u,m) \mapsto P^*_{u,m}$ is continuous, which implies our claim by composition with the continuous map $u \mapsto (u,\lambda_u)$. By  \cite[Proposition 5.10(b)]{nicole1987compactification}, the set valued map $\RC_{u,\delta_x}$ is continuous in $(u,x) \in [0,1] \times \R^d$. By Berge's theorem \cite[Theorem 17.31]{aliprantisborder} and continuity of $P \mapsto \langle P,\Gamma^{\mu_\cdot}\rangle$, the set-valued map $\RC_{u,x}^* := \arg\!\max_{P \in \RC_{u,\delta_x}}\langle P,\Gamma^{\mu_\cdot}\rangle$ has closed graph. We have just shown it to in fact be singleton-valued, or $\RC_{u,x}^*=\{P^*_{u,\delta_x}\}$ for each $(u,x)$. That is, the function $(u,x) \mapsto P^*_{u,\delta_x}$ has closed graph and is thus continuous. To conclude, simply note that $P^*_{u,m} = \int_{\R^d} P^*_{u,\delta_x}\,m(dx)$ (e.g., by \cite[Theorem 5.11(c)]{nicole1987compactification}).
\end{proof}

\begin{corollary} \label{co:cont-disint}
Suppose the assumptions of Lemma \ref{le:continuity-alpha_u} hold. Assume $\mu_\cdot$ is a W-equilibrium, with equilibrium control $\alpha^*$. Then the disintegration $[0,1] \ni u \mapsto \L(X^{\alpha^*}\,|\,U=u)\in\P(\C^d)$ admits a weakly continuous version.
\end{corollary}
\begin{proof}
Let $\alpha$ denote the equilibrium control corresponding to $\mu_\cdot$. We note again that $u \mapsto \L(X^{\lambda_u,\alpha_u})$ is a version of the conditional law $\L(X^{\alpha^*}\,|\,U=u)$.
By Lemma \ref{le:optimality_alpha_u},  the control $\alpha_u(t,x) := \alpha(t,u,x)$ optimizes $J_W^{u,\lambda_u}(\mu_\cdot,\cdot)$ over $\A_1$, for a.e.\ $u \in [0,1]$. By Lemma \ref{le:continuity-alpha_u}, there is a unique (up to Lebesgue a.e.\ equality) optimizer $\alpha^*_u \in \A_1$ of $J_W^{u,\lambda_u}(\mu_\cdot,\cdot)$. Hence $\alpha_u=\alpha^*_u$ for a.e.\ $u$, and we deduce that $\L(X^{\lambda_u,\alpha_u})=\L(X^{\lambda_u,\alpha^*_u})$ for a.e.\ $u$. The claim now follows from the last statement of Lemma \ref{le:continuity-alpha_u}.
\end{proof}

\section{Convergence of empirical measures} \label{sec:conv_emp_meas}

In preparation for Section \ref{se:approxEQ}, which proves our results about approximate equilibria, we study in this section the general principles underlying these results. 
These results deal with the convergence of neighborhood empirical measures, under various assumptions on the underlying distributions and kernel.
We work throughout this section with a general Polish space $E$.
Recall the notation $I^n_i$ from \eqref{def:stepgraphon}.

\subsection{General kernels} \label{se:emp-generalW} 
Let $(U,X)$ be a random variable taking values in $[0,1] \times E$, with law $\mu\in\Punif([0,1] \times E)$. Let $n \in \N$, and let $I^n_i = [(i-1)/n,i/n)$ as before for $i \in [n]$.
For each $n \in \N$ let $U^n_i\sim\mathrm{Unif}(I^n_i)$, and with
\begin{align*}
\L(X^n_i\,|\,U^n_i=u) = \L(X\,|\,U=u), \quad u \in I^n_i.
\end{align*}
In other words, the law of $(U^n_i,X^n_i)$ is the conditional law of $(U,X)$ given $\{U\in I^n_i\}$.
This entails in particular that, for bounded measurable $h : [0,1] \times E \to \R$,
\begin{align}
\langle\mu,h\rangle = \E[h(U,X)] &=  \sum_{i=1}^n \int_{I^n_i}\E[h(u,X)\,|\,U=u]\,du \nonumber  \\
	&= \sum_{i=1}^n \int_{I^n_i}\E[h(u,X^n_i)\,|\,U^n_i=u]\,du  \nonumber \\
	&= \frac{1}{n}\sum_{i=1}^n \E[h(U^n_i,X^n_i)].  \label{general-keyidentity}
\end{align}
Assume $(U,X)$ and $(U^n_i,X^n_i)_{i=1}^n$ are defined on the same probability space and are independent.
Let $W \in L^1_+[0,1]^2 $, and recall the definition of $\textsf{W}\mu(u)$ from \ref{def:Woperator-measure};
with $(U,X) \sim \mu$, note that we may write $\langle \textsf{W}\mu(u),\varphi\rangle = \E[W(u,U)\varphi(X)]$ for bounded measurable $\varphi$.
Let $(\xi^n_{ij})$ again be a $n \times n$ matrix with values in $[0,1]$ and with zeros on its diagonal. 
Recall that $W_{\xi^n}$ denotes the associated step kernel, as in  \eqref{def:stepgraphon}; we will use repeatedly the fact that $W_{\xi^n}(U^n_i,U^n_j)=\xi^n_{ij}$. Define lastly the (random) empirical measures
\begin{align}
M^n_i := \frac{1}{n}\sum_{j=1}^n \xi^n_{ij} \delta_{X^n_j} = \frac{1}{n}\sum_{j=1}^n W_{\xi^n}(U^n_i,U^n_j) \delta_{X^n_j}. \label{def:Mni-general}
\end{align}
Recall the definition of the strong operator topology from Section \ref{se:graphon-operators}.
The main result of this section is the following theorem, which we will apply only in cases where $h(u,x,m)$ does not depend on $x$, but the proof of the general case given here is not any more difficult.

\begin{theorem} \label{th:generalprinciple}
Assume $W_{\xi^n}$ converges to $W$ in the strong operator topology, and assume \ref{asmp:generalprinciple-A} holds.
Let $h : [0,1] \times E \times \M_+(E) \to \R$ be a bounded measurable function such that $h(u,x,\cdot)$ is continuous on $\M_+(E)$ for each fixed $(u,x) \in [0,1] \times E$. Then 
\begin{align}
\frac{1}{n}\sum_{i=1}^n\E\left[ h\left(U^n_i,X^n_i,M^n_i\right)\right] \to \E[h(U,X,
\text{\rm\textsf{W}}\mu(U))] \label{goal1}
\end{align}
\end{theorem}
\begin{proof}

We will use several times the following fact: There exists $\overline{A} > 0$ such that
\begin{align}
\frac{1}{n^2}\sum_{i,j=1}^n \xi^n_{ij} = \|W_{\xi^n}\|_{L^1[0,1]^2} \le \overline{A}, \quad \forall n \in \N. \label{pf:gen-cont-1}
\end{align}
To see this, note that the convergence in strong operator topology $W_{\xi^n} \to W$ implies 
\begin{align*}
\big|\|\textsf{W}_{\xi^n}\bm{1}\|_{L^1[0,1]} - \|\textsf{W}\bm{1}\|_{L^1[0,1]}\big| \le \|(\textsf{W}_{\xi^n}-\textsf{W})\bm{1}\|_{L^1[0,1]} \to 0,
\end{align*} 
where $\bm{1}$ is the constant function equal to 1.
Since $W$ and $W_{\xi^n}$ are nonnegative, we have $\|\textsf{W} \bm{1}\|_{L^1[0,1]} = \|W \|_{L^1[0,1]^2}$ and $\|\textsf{W}_{\xi^n}\bm{1}\|_{L^1[0,1]} = \|W_{\xi^n}\|_{L^1[0,1]^2} = \frac{1}{n^2}\sum_{i,j=1}^n \xi^n_{ij}$.

The proof proceeds by a series of simplifications.

{\ }

\noindent\textbf{Step 1.} We first argue that it suffices to prove \eqref{goal1} for $h$ bounded and 1-Lipschitz. Indeed, suppose this is the case. Define the following probability measures on $[0,1] \times E \times \M_+(E)$:
\begin{align*}
Q_n := \frac{1}{n}\sum_{i=1}^n \L(U^n_i,X^n_i,M^n_i), \qquad Q := \L(U,X,M(U)).
\end{align*}
We have assumed that $\langle Q_n,h\rangle \to \langle Q,h\rangle$ holds for bounded Lipschitz $h$. 
By the Portmanteau theorem, it also holds for bounded continuous $h$, and in particular we have $Q_n \to Q$ weakly.  The $[0,1] \times E$-marginals of $Q_n$ are all the same, i.e.,  $\frac{1}{n}\sum_{i=1}^n\L(U^n_i,X^n_i)=\L(U,X)=\mu$, for each $n$, as argued in \eqref{general-keyidentity}. Hence, the weak convergence $Q_n \to Q$ also implies the convergence $\langle Q_n,h\rangle \to \langle Q,h\rangle$ for test functions $h$ of the form in the statement of the theorem, with no continuity required in the first two arguments \cite[Lemma 2.1]{beiglbock2018denseness}.

{\ }

\noindent\textbf{Step 2.} We next claim that it suffices to show that
$\textsf{W}_{\xi^n}\mu_n(U) \to \textsf{W}\mu(U)$ in probability, where the random probability measure $\mu_n$ on $[0,1] \times E$ is defined by
\begin{align*}
\mu_n = \frac{1}{n}\sum_{i=1}^n \delta_{(U^n_i,X^n_i)}.
\end{align*}
Expanding the notation and applying the definition \eqref{def:Woperator-measure} of the operator $\textsf{W}_{\xi^n}$,
\begin{align*}
\textsf{W}_{\xi^n}\mu_n(u) = \frac{1}{n}\sum_{j=1}^n W_{\xi^n}(u,U^n_j) \delta_{X^n_j} = M^n_i, \quad \text{for } u \in I^n_i, \ i=1,\ldots,n.
\end{align*}
Recalling that $W_{\xi^n}(u,U^n_j)=W_{\xi^n}(U^n_i,U^n_j)$ for $u \in I^n_i$, we have
\begin{align*}
\frac{1}{n}\sum_{i=1}^n\E[ h(U^n_i,X^n_i,M^n_i)] &= \frac{1}{n}\sum_{i=1}^n \E\Big[ h\Big(U^n_i,X^n_i,\frac{1}{n}\sum_{j=1}^n W_{\xi^n}(U^n_i,U^n_j) \delta_{X^n_j}\Big) \Big] \\
	&=   \sum_{i=1}^n \int_{I^n_i} \E\Big[ h\Big(u,X^n_i,\frac{1}{n}\sum_{j=1}^n W_{\xi^n}(u,U^n_j) \delta_{X^n_j}\Big) \,\Big|\, U^n_i=u \Big]\,du,
\end{align*}
with the second step using independence of $(U^n_i,X^n_i)_{i=1}^n$ and the fact that $W_{\xi^n}(u,U^n_j)=W_{\xi^n}(U^n_i,U^n_i)=\xi^n_{ii}=0$ for $u \in I^n_i$. Since $\L(X^n_i\,|\,U^n_i=u)=\L(X\,|\,U=u)$ for $u \in I^n_i$, this simplifies to
\begin{align*}
\qquad\qquad &=   \sum_{i=1}^n \int_{I^n_i} \E\Big[ h\Big(u,X,\frac{1}{n}\sum_{j=1}^n W_{\xi^n}(u,U^n_j) \delta_{X^n_j}\Big) \,\Big|\, U=u \Big]\,du \\
	&= \int_0^1 \E[ h(u,X ,\textsf{W}_{\xi^n}\mu_n(u)) \,|\,U=u]\,du \\
	&= \E[ h(U,X,\textsf{W}_{\xi^n}\mu_n(U))].
\end{align*}
Here we used also the assumed independence of $(U,X)$ and $(U^n_i,X^n_i)_{i=1}^n$.
Hence, once we know that $\textsf{W}_{\xi^n}\mu_n(U) \to \textsf{W}\mu(U)$ in probability,
it follows from the bounded convergence theorem that \eqref{goal1} holds for bounded continuous $h$, which is sufficient by Step 1.

{\ }

\noindent\textbf{Step 3.} We finally prove that $\textsf{W}_{\xi^n}\mu_n(U) \to \textsf{W}\mu(U)$ in probability, which will complete the proof as explained in Step 2. Fix a bounded continuous function $\varphi : E \to [-1,1]$. 
Expanding the definition,
\begin{align*}
\langle \textsf{W}_{\xi^n}\mu_n(u),\varphi\rangle  &= \frac{1}{n} \sum_{j=1}^n W_{\xi^n}(u,U^n_j)\varphi(X^n_j).
\end{align*}
We must show that $\langle \textsf{W}_{\xi^n}\mu_n(U),\varphi\rangle \to \langle \textsf{W}\mu(U),\varphi\rangle$ in probability.

{\ }

\noindent\textbf{Step 3a.}
We first claim that $\langle \textsf{W}_{\xi^n}\mu_n(U),\varphi\rangle - \E[\langle \textsf{W}_{\xi^n}\mu_n(U),\varphi\rangle \,|\,U] \to 0$ in probability, and in fact in $L^2$. To see this, note for $u \in I^n_i$ that
\begin{align*}
\Var(\langle \textsf{W}_{\xi^n}\mu_n(U),\varphi\rangle \,|\, U=u) &= \Var \Big(\frac{1}{n} \sum_{j=1}^n \xi^n_{ij}\varphi(X^n_j)  \Big) \le \frac{1}{n^2} \sum_{j=1}^n (\xi^n_{ij})^2,
\end{align*} 
by independence of $(X^n_j)_{j=1}^n$. Hence,
\begin{align*}
\E\left[\left(\langle \textsf{W}_{\xi^n}\mu_n(U),\varphi\rangle - \E[\langle \textsf{W}_{\xi^n}\mu_n(U),\varphi\rangle \,|\,U]\right)^2\right] &= \E \, \Var(\langle \textsf{W}_{\xi^n}\mu_n(U),\varphi\rangle \,|\, U) \\
	&= \sum_{i=1}^n \int_{I^n_i} \Var(\langle \textsf{W}_{\xi^n}\mu_n(U),\varphi\rangle \,|\, U=u)\, du \\
     &\le \frac{1}{n^3} \sum_{i,j=1}^n (\xi^n_{ij})^2,
\end{align*}
which vanishes by \eqref{asmp:generalprinciple-A}.

{\ }

\noindent\textbf{Step 3b.}
We must finally show that $\E[\langle \textsf{W}_{\xi^n}\mu_n(U),\varphi\rangle \,|\,U] \to \langle \textsf{W}\mu(U),\varphi\rangle$ in probability. To see this, we first use again the independence of $(U^n_i,X^n_i)_{i=1}^n$ to rewrite 
\begin{align*}
\E[\langle \textsf{W}_{\xi^n}\mu_n(U),\varphi\rangle \,|\,U=u] &= \E\Big[\frac{1}{n}\sum_{i=1}^nW_{\xi^n}(u,U^n_i)\varphi(X^n_i)\Big] = \E[ W_{\xi^n}(u,U)\varphi(X)] \\
	&= \int_0^1 W_{\xi^n}(u,v)\psi(v)\,dv,
\end{align*}
where $\psi(v):= \E[\varphi(X)\,|\,U=v]$, and where we again used the fact that $\frac{1}{n}\sum_{i=1}^n\L(U^n_i,X^n_i)=\L(U,X)$ as shown by \eqref{general-keyidentity}. Similarly, we may write
\begin{align*}
\langle \textsf{W}\mu(u),\varphi\rangle &= \E[W(u,U)\varphi(X)] = \E[W(u,U)\psi(U)] = \int_0^1 W(u,v)\psi(v)\,dv.
\end{align*}
These identities are to be understood for a.e.\ $u \in [0,1]$, and combined they yield
\begin{align}
\E\big[\big|\E[\langle \textsf{W}_{\xi^n}\mu_n(U),\varphi\rangle \,|\,U] - \langle \textsf{W}\mu(u),\varphi\rangle \big|\big] &= \int_0^1 \bigg|\int_0^1 \big(W_{\xi^n}(u,v) - W(u,v)\big)\psi(v)\,dv\bigg|\,du \nonumber \\
	&= \|(\textsf{W}_{\xi^n}-\textsf{W})\psi\|_{L^1[0,1]}, \label{pf:general1}
\end{align}
where we have used the operator notation of \eqref{def:Woperator-function}.

Recalling that $\varphi$ and thus $\psi$ are bounded, the right-hand side of \eqref{pf:general1} converges to zero by the assumption that $W_{\xi^n}\to W$ in the strong operator topology. 
We deduce that $\E[\langle \textsf{W}_{\xi^n}\mu_n(U),\varphi\rangle \,|\,U] \to \langle \textsf{W}\mu(U),\varphi\rangle$ in $L^1$ and thus in probability. 
This completes the proof of Step 3b, and thus the theorem.
\end{proof}

\subsection{Continuous kernels} \label{se:emp-contW}
We now prove an alternative to Theorem \ref{th:generalprinciple} which requires stronger assumptions but is, in a sense, uniform in the choice of labels, rather than averaged. Fix again $\mu\in\Punif([0,1] \times E)$, and assume there exists a version of the disintegration $\mu(du,dx)=du\mu_u(dx)$ such that $[0,1] \ni u \mapsto \mu_u \in \P(E)$ is weakly continuous. For $u \in [0,1]$, let $X_u$ denote a random variable with law $\mu_u$.
Let us write $\bm{u}=(u_1,\ldots,u_n)$ for a generic element of  $I^n_1 \times \cdots \times I^n_n$, which we think of as denoting the set of admissible assignments of \emph{labels} to each player $i \in [n]$. For $\bm{u} \in I^n_1 \times \cdots \times I^n_n$ define the (random) empirical measures
\begin{align}
\Mui := \frac{1}{n}\sum_{j=1}^n \xi^n_{ij} \delta_{X_{u_j}} = \frac{1}{n}\sum_{j=1}^n W_{\xi^n}(u_i,u_j) \delta_{X_{u_j}}, \label{def:Mni-cont}
\end{align}
where $(X_{u_i})_{i=1}^n$ are assumed independent.
Let us stress that \eqref{def:Mni-cont} and every other expression below will involve at most finitely many of the random variables  $(X_u)_{u \in [0,1]}$ at a time; at no point must we face any of the complications that accompany a continuum of independent random variables.

Recall below the bounded Lipschitz norm $\|\cdot\|_{BL}$ defined in \eqref{bounded-Lipschitz-norm}.

\begin{theorem} \label{th:generalprinciple-cont}
Assume $W_{\xi^n}$ converges to $W$ in the strong operator topology, and assume that \eqref{asmp:generalprinciple-A} holds.
Assume also that $[0,1] \ni u \mapsto W(u, v)dv \in \M_+([0,1])$ is continuous, and that there exists a version of the disintegration $\mu(du,dx)=du\mu_u(dx)$ such that the map $[0,1] \ni u \mapsto \mu_u \in \P(E)$ is weakly continuous.

Then 
\begin{align}
\lim_{n \to \infty}\sup_{\bm{u} = (u_1, ..., u_n) \in I^n_1 \times ... \times I^n_n} \frac{1}{n} \sum_{i=1}^n \E \|  \Mui - \text{\rm\textsf{W}}\mu(u_i) \|_{BL} = 0. \label{goal1-cont}
\end{align}
Let $h : [0,1] \times \M_+(E) \to \R$ be bounded and measurable, and assume $h(u,\cdot)$ continuous on $\M_+(E)$ uniformly in $u \in [0,1]$, in the sense that
\begin{align*}
\lim_{m' \to m}\sup_{u \in [0,1]}|h(u,m')-h(u,m)| = 0, \quad \forall m \in \M_+(E).
\end{align*}
Then we have
\begin{align}
\lim_{n \to \infty}\sup_{\bm{u} = (u_1, ..., u_n) \in I^n_1 \times ... \times I^n_n} \frac{1}{n} \sum_{i=1}^n \E|h(u_i, \Mui) - h(u_i, \text{\rm\textsf{W}}\mu(u_i))| = 0. \label{def:hform-cont}
\end{align}
\end{theorem}
\begin{proof}
The claim \eqref{def:hform-cont} follows immediately from \eqref{goal1-cont} and the assumed uniform continuity of $h$.
As in the proof of Theorem \ref{th:generalprinciple}, the convergence in cut norm $W_{\xi^n} \to W$ yields $\overline{A} > 0$ such that \eqref{pf:gen-cont-1} holds.

{\ } 

\noindent\textbf{Step 1.} We first prove that 
\begin{align}
\lim_{n \to \infty}  \sup_{\bm{u} = (u_1, ..., u_n) \in I^n_1 \times ... \times I^n_n}  \frac{1}{n} \sum_{i=1}^n \E | \langle \Mui - \textsf{W}\mu(u_i), \varphi \rangle | = 0, \label{goal2-cont}
\end{align}
for each Lipschitz function $\varphi : E \to [-1,1]$.
Note first for each $ i \in [n]$ and $\bm{u} \in I^n_1 \times ... \times I^n_n$ that
\begin{equation*}
\E | \langle \Mui - \textsf{W}\mu(u_i), \varphi \rangle | \le \E | \langle \Mui - \E \Mui, \varphi \rangle | + | \langle \E \Mui - \textsf{W}\mu(u_i), \varphi \rangle |.
\end{equation*}
For the first term, note that
\begin{align*}
\big(\E | \langle \Mui - \E \Mui, \varphi \rangle |\big)^2 &\le \Var( \langle \Mui , \varphi \rangle ) = \Var\bigg(\frac{1}{n}\sum_{j=1}^n \xi^n_{ij} \varphi(X_{u_j})\bigg)  \le \frac{1}{n^2}\sum_{j=1}^n(\xi^n_{ij})^2.
\end{align*}
Using the assumption \eqref{asmp:generalprinciple-A}, we deduce
\begin{align*}
\sup_{\bm{u} \in I^n_1 \times ... \times I^n_n}\frac{1}{n}\sum_{i=1}^n \E | \langle \Mui - \E \Mui, \varphi \rangle | \le \bigg(\frac{1}{n^3}\sum_{i,j=1}^n(\xi^n_{ij})^2\bigg)^{1/2} \to 0,
\end{align*}
and thus \eqref{goal2-cont} will follow if we show that
\begin{align}
\lim_{n\to\infty} \sup_{\bm{u} \in I^n_1 \times ... \times I^n_n}\frac{1}{n}\sum_{i=1}^n | \langle \E \Mui - \textsf{W}\mu(u_i), \varphi \rangle | = 0. \label{pf:gen-cont-2}
\end{align}
Fix $i \in [n]$ and  $\bm{u} \in I^n_1 \times ... \times I^n_n$ for now.
Using $\L(X_{u_j}) = \mu_{u_j}$ and the fact that $\xi^n_{ij} = W_{\xi^n}(u_i, u_j) = n \int_{I^n_j} W_{\xi^n}(u_i, v) dv$, we have on the one hand
\begin{align*}
\E\langle \Mui,\varphi\rangle &= \frac{1}{n}\sum_{j=1}^n\xi^n_{ij}\E[\varphi(X_{u_j})] = \sum_{j=1}^n \int_{I^n_j} W_{\xi^n}(u_i,v)\langle \mu_{u_j},\varphi\rangle \, dv.
\end{align*}
On the other hand,
\begin{align*}
\langle \textsf{W}\mu(u_i),\varphi\rangle &=  \int_{[0,1] \times E} W(u_i,v)\varphi(x) \mu(dv,dx) = \int_0^1 W(u_i,v)\langle \mu_v,\varphi\rangle \, dv.
\end{align*}

Hence, to prove \eqref{pf:gen-cont-2}, we must show equivalently that
\begin{equation}\label{eq:avg_graphon_det_u}
\lim_{n\to\infty} \sup_{\bm{u} \in I^n_1 \times ... \times I^n_n}\frac{1}{n} \sum_{i=1}^n \bigg| \sum_{j=1}^n \int_{I^n_j} W_{\xi^n}(u_i,v)\langle \mu_{u_j},\varphi\rangle\,dv - \int_0^1 W(u_i,v)\langle \mu_v,\varphi\rangle \, dv \bigg| = 0.
\end{equation}
To prove this, we split the difference into three terms:
\begin{equation}\label{eq:diff_Wn_mu_u^n_j-W_mu_v}
\begin{split}
   \frac{1}{n} &\sum_{i=1}^n \bigg| \sum_{j=1}^n \int_{I^n_j}  W_{\xi^n}(u_i,v)\langle \mu_{u_j},\varphi\rangle \, dv - \int_0^1 W(u_i,v)\langle \mu_v,\varphi\rangle \, dv \bigg| \\
    &\hspace{.5cm} \le \frac{1}{n} \sum_{i=1}^n \bigg| \sum_{j=1}^n \int_{I^n_j} \big(W_{\xi^n}(u_i,v)\langle \mu_{u_j},\varphi\rangle - W_{\xi^n}(u_i,v)\langle \mu_v,\varphi\rangle\big)\, dv \bigg|\\
    & \hspace{1cm} + \frac{1}{n} \sum_{i=1}^n \bigg|\int_0^1 W_{\xi^n}(u_i,v)\langle \mu_v,\varphi\rangle dv - n \int_{I^n_i} \int_0^1 W(u,v) \langle \mu_v,\varphi\rangle dv du \bigg| \\
    & \hspace{1cm} + \frac{1}{n} \sum_{i=1}^n \bigg|n \int_{I^n_i} \int_0^1 W(u,v) \langle \mu_v,\varphi\rangle dv du - \int_0^1 W(u_i,v)\langle \mu_v,\varphi\rangle dv \bigg|
\end{split}
\end{equation}
By definition of the step graphon $W_{\xi^n}$, the first term is equal to 
\begin{align}
    \frac{1}{n} \sum_{i=1}^n \left| \sum_{j=1}^n \xi^n_{ij} \int_{I^n_j} \langle \mu_{u_j}- \mu_v,\varphi\rangle \, dv \right| &\le \frac{1}{n} \sum_{i,j=1}^n \xi^n_{ij} \int_{I^n_j} \left|\langle \mu_{u_j},\varphi\rangle - \langle \mu_v,\varphi\rangle\right|\, dv. \label{pf:gen-cont-term1}
\end{align}
We deduce from the assumption of weak continuity of $u \mapsto \mu_u$ that $[0,1] \ni u \mapsto \langle \mu_u, \varphi \rangle \in \R$ is uniformly continuous. For a given $\epsilon >0$, we can therefore choose $n$ large enough so that  $|\langle \mu_u - \mu_v, \varphi \rangle | \le \epsilon$ whenever $|u-v| \le 1/n$. Hence, for large enough $n$ not depending on the choice of $\bm{u}$ we find that the right-hand side of \eqref{pf:gen-cont-term1} is bounded by $\overline{A}\epsilon$.

Having dealt with the first term in \eqref{eq:diff_Wn_mu_u^n_j-W_mu_v}, let us turn to the second. Using the fact that $ W_{\xi^n}(u_i, v) = n \int_{I^n_i} W_{\xi^n}(u, v) du$, we can rewrite it as 
\begin{equation*}
\begin{split}
\sum_{i=1}^n &\bigg|\int_{I^n_i} \int_0^1 \!\!\big(W_{\xi^n}(u,v) - W(u,v)\big)\langle \mu_v,\varphi\rangle dv du \bigg|  \le  \int_0^1 \!\bigg| \!\int_0^1 \!\!\big(W_{\xi^n}(u,v) - W(u,v)\big)\langle \mu_v,\varphi\rangle dv \bigg|du .
\end{split}
\end{equation*}
Since $\varphi$ is bounded, the right-hand side (which we note does not depend on $\bm{u}^n$) converges to zero by the assumption that $W_{\xi^n}\to W$ in the strong operator topology.

Finally, the third term in \eqref{eq:diff_Wn_mu_u^n_j-W_mu_v} is equal to
\begin{equation}
    \sum_{i=1}^n \bigg| \int_{I^n_i} (\psi(u)- \psi(u_i))du \bigg|, \label{pf:gen-cont-term3}
\end{equation}
where we define $\psi(u) = \int_0^1 W(u,v)\langle \mu_v,\varphi\rangle dv$. Recall by assumption that $u \mapsto W(u,v)dv \in \M_+([0,1])$ is continuous. Since $v \mapsto \langle \mu_v,\varphi\rangle$ is continuous by assumption, we deduce that $\psi$ is continuous. Therefore, given $\epsilon > 0$, we may choose $n$ large enough so that $|\psi(u)- \psi(v)| \le \epsilon$ whenever $|u-v|\le 1/n$, and it follows that \eqref{pf:gen-cont-term3} is no more than $\epsilon$, regardless of the choice of $\bm{u} \in I^n_1 \times ... \times I^n_n$. This concludes the proof of \eqref{goal2-cont}.

{\ }

\noindent\textbf{Step 2.} We next show that the set of mean measures $\{\frac{1}{n}\sum_{i=1}^n\E \Mui : n \ge 1, \, \bm{u} \in I^n_1 \times \cdots \times I^n_n\} \subset \M_+(E)$ is tight. The mean measures are given by
\begin{align*}
\frac{1}{n}\sum_{i=1}^n\E \Mui  &= \frac{1}{n^2}\sum_{i,j=1}^n \xi^n_{ij}\L(X_{u_j}) = \frac{1}{n^2}\sum_{i,j=1}^n \xi^n_{ij}\mu_{u_j}.
\end{align*}
Since the map $u \mapsto \mu_u$ is continuous by assumption, the image $\{\mu_u : u \in [0,1]\} \subset \P(E)$ is compact and thus tight by Prokhorov's theorem. Hence, for $\epsilon > 0$, we may find $K \subset E$ compact such that $\mu_u(K^c) \le \epsilon$ for all $u \in [0,1]$. By \eqref{pf:gen-cont-1},  $\frac{1}{n}\sum_{i=1}^n\E \Mui(K^c) \le \overline{A}\epsilon$.

{\ }

\noindent\textbf{Step 3.} We now prove the claim \eqref{goal1-cont}. 
Let $\mathcal{S}$ denote the set of 1-Lipschitz functions $\varphi : E \to [-1,1]$, and let $\epsilon > 0$. 

By Lemma \ref{le:Wcontinuity}(3), the continuity assumptions on $W$ and the disintegration $\mu_u$ imply that the map $[0,1] \ni u \mapsto \textsf{W}\mu(u) \in \M_+(E)$ is continuous, and thus the set of  measures $\{\textsf{W}\mu(u): u \in [0,1]\} \subset \M_+(E)$ is tight.
This and Step 2 imply that there exists a compact set $K \subset E$ such that 
\begin{align}
\sup_{u \in [0,1]} \textsf{W}\mu(u)(K^c) + \sup_{n \in \N}\sup_{\bm{u} \in I^n_1 \times \cdots \times I^n_n} \frac{1}{n}\sum_{i=1}^n \E \Mui(K^c) \le \epsilon. \label{pf:gen-cont-11}
\end{align}
The set of 1-Lipschitz functions $K \to [-1,1]$ is compact in the uniform topology, by Arzel\`a-Ascoli. We may thus find a finite set $\mathcal{S}_\epsilon \subset \mathcal{S}$ such that $\min_{\psi \in \mathcal{S}_\epsilon}\|(\varphi-\psi)1_K\|_\infty \le \epsilon$ for every $\varphi \in \mathcal{S}$. Now, for any $\varphi,\psi \in \mathcal{S}$ and $u \in [0,1]$, we have
\begin{align*}
|\langle \Mui-\textsf{W}\mu(u),\varphi\rangle| &\le |\langle \Mui-\textsf{W}\mu(u),\psi\rangle| + |\langle \Mui-\textsf{W}\mu(u),(\varphi -\psi) 1_{K^c}\rangle| \\
	&\quad + |\langle \Mui-\textsf{W}\mu(u),(\varphi -\psi) 1_K\rangle|.
\end{align*}
To estimate the second and third terms, we argue that the total masses of the measures $\frac{1}{n}\sum_{i=1}^n\Mui$ and $\textsf{W}\mu(u)$ are bounded a.s.\ by some constant $C>0$. Indeed, $\frac{1}{n}\sum_{i=1}^n\Mui(E) \le \overline{A}$ a.s.\ by \eqref{pf:gen-cont-1}, and the mass $\textsf{W}\mu(u)(E)=\langle \textsf{W}\mu(u),1\rangle$ depends continuously on $u$ thanks to Lemma \ref{le:Wcontinuity}(3) and the assumed continuity of $W$. Hence,  for $\bm{u} \in I^n_1 \times \cdots \times I^n_n$,
\begin{align*}
\frac{1}{n}\sum_{i=1}^n\|\Mui &- \textsf{W}\mu(u_i)\|_{BL} = \frac{1}{n}\sum_{i=1}^n\sup_{\varphi \in \mathcal{S}}|\langle \Mui-\textsf{W}\mu(u_i),\varphi\rangle| \\
	&\le \frac{1}{n}\sum_{i=1}^n \Big[\max_{\psi \in \mathcal{S}_\epsilon} |\langle \Mui-\textsf{W}\mu(u_i),\psi\rangle| + 2\Mui(K^c) + 2\textsf{W}\mu(u_i)(K^c)\Big] + 2C\epsilon.
\end{align*}
Take expectations, recalling \eqref{pf:gen-cont-11}, and bound $\max_{\psi \in \mathcal{S}_\epsilon}$ by $\sum_{\psi \in \mathcal{S}_\epsilon}$ to get
\begin{align*}
\frac{1}{n}\sum_{i=1}^n\E\|\Mui - \textsf{W}\mu(u_i)\|_{BL} &\le \sum_{\psi \in \mathcal{S}_\epsilon} \frac{1}{n}\sum_{i=1}^n\E|\langle \Mui-\textsf{W}\mu(u_i),\psi\rangle| + 2(2+C)\epsilon,
\end{align*}
for all $i \in [n]$ and all $\bm{u} \in I^n_1 \times \cdots \times I^n_n$.
Send $n\to\infty$ followed by $\epsilon \to 0$ to deduce  \eqref{goal1-cont}.
\end{proof}

\begin{remark}
Theorem \ref{th:generalprinciple-cont} remains valid under a somewhat weaker convergence assumption than strong operator topology, namely that $\|(\textsf{W}_{\xi^n} - \textsf{W})\psi\|_{L^1[0,1]} \to 0$ for $\psi \in C[0,1]$, not necessarily for all $\psi\in L^\infty[0,1]$. This is, of course, what one would call the strong operator topology for the space of operators from $C[0,1] \to L^1[0,1]$. In fact, we do not really need the limit operator $\textsf{W}$ to be an integral operator; it could be something of the form $\textsf{W}\varphi(u)=\int_{[0,1]}\varphi(v)K_u(dv)$ for some measurable map $u \mapsto K_u \in \M_+(E)$ with $\int_0^1K_u(E)\,du < \infty$.
This is somewhat similar to the (more subtle) notion of \emph{extended graphons} used in the recent study \cite{jabin2021mean} of (non-game-theoretic) interacting diffusions, but we will not pursue this generality here.
\end{remark}

\subsection{Sampling kernels} \label{se:emp-sampling}

The mode of convergence can be further upgraded under the more specific choice of graphon adopted in Theorem \ref{th:approxEQ-sampling}. Rather than working with a generic matrix $\xi^n$ such that $W_{\xi^n} \to W$, let us now follow a canonical construction in graphon theory. In this section, let us define the empirical measure
\begin{align*}
N_i^{n,\bm{u}} = \frac{1}{n}\sum_{j=1,\,j \neq i}^n W(u_i,u_j) \delta_{X_{u_j}}, \quad \text{for } \bm{u}=(u_1,\ldots,u_n) \in [0,1]^n, \ n \in \N,
\end{align*}
where $X_u\sim \mu_u$ for each $u$ are independent as in Section \ref{se:emp-contW}.
In the following, equip $[0,1]^\infty$ with the infinite product measure $(\mathrm{Unif}[0,1])^\infty$.

\begin{theorem} \label{th:generalprinciple-sampling}
Assume $W : [0,1]^2 \to [0,\infty)$ is bounded and measurable. Assume $\{ \mu_u : u \in [0,1]\} \subset \P(E)$ is tight. Let $h : [0,1] \times \M_+(E) \to \R$ be bounded and measurable, and assume $h(u,\cdot)$ continuous on $\M_+(E)$ uniformly in $u \in [0,1]$, in the sense that
\begin{align*}
\lim_{m' \to m}\sup_{u \in [0,1]}|h(u,m')-h(u,m)| = 0, \quad \forall m \in \M_+(E).
\end{align*}
Then, for almost every choice of $(u_i)_{i \in \N} \in [0,1]^\infty$, the following holds:
\begin{align}
\lim_{n\to\infty} \max_{i \in [n]}\E\left| h(u_i,N_i^{n,(u_1,\ldots,u_n)}) - h(u_i,\text{\rm\textsf{W}}\mu(u_i))\right| = 0. \label{goal111}
\end{align}
\end{theorem}
\begin{proof}
By rescaling, we may assume that $0 \le W \le 1$ and $0 \le h \le 1$. Let $(u_i)_{i \in \N}$ be arbitrary for now. Let $\varphi : E \to [0,1]$ be measurable, and set $\psi(u)=\E[\varphi(X_u)]$. By the union bound and Hoeffding's inequality,
\begin{align*}
\PP\bigg(\max_{i \in [n]}\bigg|\frac{1}{n}\sum_{j=1,\,j \neq i}^n W(u_i,u_j) \varphi(X_{u_j}) - \frac{1}{n}\sum_{j=1,\,j \neq i}^n W(u_i,u_j) \psi(u_j)\bigg| > \delta\bigg) &\le n e^{-2n\delta^2},
\end{align*}
for each $n \in \N$ and $\delta > 0$. By Borel-Cantelli, we deduce
\begin{align}
\max_{i \in [n]}\bigg|\frac{1}{n}\sum_{j=1,\,j \neq i}^n W(u_i,u_j) \varphi(X_{u_j}) - \frac{1}{n}\sum_{j=1,\,j \neq i}^n W(u_i,u_j) \psi(u_j)\bigg| \to 0, \ \ a.s. \label{pf:a.s.1}
\end{align}

Next, let $U_i \sim \mathrm{Unif}[0,1]$ for $i \in \N$ be i.i.d. Again using Hoeffding's inequality, we find
\begin{align*}
\PP\bigg( \bigg|\frac{1}{n}\sum_{j=1,\,j \neq i}^n W(U_i,U_j) \psi(U_j) - \frac{1}{n}\sum_{j=1,\,j \neq i}^n \E[W(U_i,U_j) \psi(U_j)\,|\,U_i]\bigg| > \delta \,\bigg|\, U_i\bigg) &\le  e^{-2n\delta^2},
\end{align*}
for each $i$, a.s. Note that $\frac{1}{n}\sum_{j=1,\,j \neq i}^n \E[W(u,U_j) \psi(U_j)] = \frac{n-1}{n}\E[W(u,U)\psi(U)]= \frac{n-1}{n}\langle \textsf{W}\mu(u),\varphi\rangle$.
Hence, for $n$ large enough that $1/n \le \delta$, we get
\begin{align*}
\PP\bigg( \bigg|\frac{1}{n}\sum_{j=1,\,j \neq i}^n W(U_i,U_j) \psi(U_j) - \langle \textsf{W}\mu(U_i),\varphi\rangle \bigg| > 2\delta \,\bigg|\, U_i\bigg) &\le  e^{-2n\delta^2}.
\end{align*}
Using a union bound and the tower property,
\begin{align*}
\PP\bigg(\max_{i \in [n]}\bigg|\frac{1}{n}\sum_{j=1,\,j \neq i}^n W(U_i,U_j) \psi(U_j) - \langle \textsf{W}\mu(U_i),\varphi\rangle\bigg| > 2\delta \bigg) &\le n e^{-2n\delta^2},
\end{align*}
again for $n \ge 1/\delta$. Deduce from Borel-Cantelli that
\begin{align}
\max_{i \in [n]}\bigg|\frac{1}{n}\sum_{j=1,\,j \neq i}^n W(U_i,U_j) \psi(U_j) - \langle \textsf{W}\mu(U_i),\varphi\rangle \bigg| \to 0, \ \ a.s. \label{pf:a.s.2}
\end{align}
Combine \eqref{pf:a.s.1} and \eqref{pf:a.s.2} to get, for instance,
\begin{align*}
\E\max_{i \in [n]}\bigg|\frac{1}{n}\sum_{j=1,\,j \neq i}^n W(u_i,u_j) \varphi(X_{u_j}) - \langle \textsf{W}\mu(u_i),\varphi\rangle\bigg| \to 0,
\end{align*}
for a.e.\ choice of $(u_i)_{i \in \N}$. Since we assumed $\{ \mu_u : u \in [0,1]\}$ to be tight, it follows easily from boundedness of $W$ that $\{ \textsf{W}\mu(u) : u \in [0,1]\} \subset \M_+(E)$ is also tight, and so is $\{\E N_i^{n,\bm{u}} : n \in \N, i \in [n], \bm{u} \in [0,1]^n\}$. The latter implies that $\{\L(N_i^{n,\bm{u}}) : n \in \N, i \in [n], \bm{u} \in [0,1]^n\} \subset \P(\M_+(E))$ is tight, by a well known argument \cite[Fact (2.5)]{sznitman1991topics} which works not only for probability measures but also for nonnegative measures of uniformly bounded total mass. We may then argue as in Step 3 of the proof of Theorem \ref{th:generalprinciple-cont} that
\begin{align*}
\E\max_{i \in [n]}\|N_i^{n,(u_1,\ldots,u_n)} -  \textsf{W}\mu(u_i) \|_{BL} \to 0,  \label{pf:a.s.3}
\end{align*}
for a.e.\ choice of $(u_i)_{i \in \N}$. We now easily deduce \eqref{goal111} using the uniform continuity assumption on $h$.
\end{proof}

\section{Approximate equilibria} \label{se:approxEQ}

In this section, we will prove the results of Section \ref{subse:main_res_approx_eq}. Recall that $\alpha^*$ denotes the given $W$-equilibrium control, $X^{\alpha^*}$ the corresponding state process, and $U \sim \mbox{Unif}[0,1]$. 

In this section, we will denote $P^{\alpha^*}=\L(U,X^{\alpha^*}) \in  \Punif([0,1] \times \C^d)$  the equilibrium joint law, where we recall that $\C^d=C([0,T];\R^d)$, which is a path space law and which will enable us to use the results proved in Section \ref{sec:conv_emp_meas}. Let $\mu_\cdot \in C([0,T]; \Punif([0,1] \times \R^d))$ represent the measure flow associated with $(U,X^{\alpha^*})$, i.e., $\mu_t := \L(U, X^{\alpha^*}_t)$ for all $t \in [0,T]$. Note that $\mu_t$ is the time-$t$ marginal of $P^{\alpha^*}$, and thus $(\textsf{W}P^{\alpha^*}(u))_t = \textsf{W}\mu_t(u)$, for each $t \in [0,T]$.

We first elaborate on the notation of Section \ref{se:finitegame}, to keep track of the labels (and thus the controls) assigned to each player.
For $n \in \N$ and $\boldsymbol{u}^n := (u^n_1, \ldots, u^n_n) \in [0,1]^n$, let $\bm{X}^{n,\bm{u}^n} := (X^{n,u^n_i,i})_{i \in [n]}$ be the process satisfying the dynamics, 
\begin{equation}\label{eq:nondeviating_process_X^n,i}
\begin{split}
       dX^{n,u^n_i,i}_t &= b(t, X^{n,u^n_i,i}_t, \alpha^*(t,u^n_i, X^{n,u^n_i,i}_t)) dt + \sigma(t, X^{n,u^n_i,i}_t) dB^i_t, \quad X^{n,u^n_i,i}_0 \sim \lambda_{u^n_i},
\end{split}
\end{equation} 
where $B^i$ are independent Brownian motions, and the initial positions $(X^{n,u^n_i,i}_0)_{i \in [n]}$ are independent.
For each $i$ and each $\beta \in \A_n$, let $X^{n,\beta,u^n_i,i}$ be the process arising when player $i$ switches from the control $\alpha^*(\cdot, u^n_i, \cdot)$ to the control $\beta$. More precisely, the process $X^{n,\beta,u^n_i,i}$ is characterized by the dynamics
\begin{equation}\label{eq:deviating_process}
\begin{split}
        dX^{n,\beta,u^n_i,i}_t &= b(t, X^{n,\beta,u^n_i,i}_t, \beta(t, \bm{X}^{n,\beta,\bm{u}^n,i}_t)) dt + \sigma(t, X^{n,\beta,u^n_i,i}_t) dB^i_t, \quad X^{n,\beta,u^n_i,i}_0 \sim \lambda_{u^n_i},
\end{split}
\end{equation}
where we write $\bm{X}^{n,\beta,\bm{u}^n,i}_t$ to denote the vector $\bm{X}^{n,\bm{u}^n}_t$ but with $i^\text{th}$ component equal to $X^{n,\beta,u^n_i,i}_t$ instead of $ X^{n,u^n_i,i}_t$. To simplify the notation, we will sometimes abbreviate $\beta_t=\beta(t, \bm{X}^{n,\beta,\bm{u}^n,i}_t)$. 
Let us write also
\begin{equation}\label{eq:M^n,u^n,i_t}
    M^{n,\bm{u}^n,i} := \frac{1}{n} \sum_{j=1}^n \xi^n_{ij} \delta_{X^{n,u^n_j,j}}, 
\end{equation}
similarly to \eqref{def:Mni-cont}, for the empirical measure appearing in the objective functions of player $i$. Note that since $\xi^n_{ii} = 0$, this empirical measure does not depend on the choice of control of player $i$, and in particular if player $i$ deviates to $\beta$ then the empirical measure \ref{eq:M^n,u^n,i_t} does not need to be modified. 

Let us introduce some notations that will guide us through the proofs.
Recalling the definition of $\epsilon^n_i$, we can bound it by three terms,
\begin{equation*}
\epsilon^n_i(\bm{u}^n) \le \sup_{\beta \in \A_n}\Delta^{n,i}_1(\beta, \bm{u}^n) + \sup_{\beta \in \A_n}\Delta^{n,i}_2(\beta,  \bm{u}^n) + \Delta^{n,i}_3(\bm{u}^n),
\end{equation*}
where we defined
\begin{equation*}
    \begin{split}
        \Delta^{n,i}_1(\beta, \bm{u}^n) &:= \E \left[\int_0^T f(t, X^{n,\beta,u^n_i,i}_t, M^{n,\bm{u}^n,i}_t, \beta_t) dt + g(X^{n,\beta,u^n_i,i}_T, M^{n,\bm{u}^n,i}_T) \right] \\
        &\hspace{.6cm}- \E \left[\int_0^T f(t, X^{n,\beta,u^n_i,i}_t, \textsf{W}\mu_t(u^n_i), \beta_t) dt + g(X^{n,\beta,u^n_i,i}_T, \textsf{W}\mu_T(u^n_i)) \right], \\
        \Delta^{n,i}_2(\beta,  \bm{u}^n) &:= \E \left[\int_0^T f(t, X^{n,\beta,u^n_i,i}_t, \textsf{W}\mu_t(u^n_i), \beta_t) dt + g(X^{n,\beta,u^n_i,i}_T, \textsf{W}\mu_T(u^n_i)) \right] \\
        &\hspace{.6cm}- \E \left[\int_0^T f(t, X^{n,u^n_i,i}_t, \textsf{W}\mu_t(u^n_i), \alpha^*(t, u^n_i, X^{n,u^n_i,i}_t)) dt + g(X^{n,u^n_i,i}_T, \textsf{W}\mu_T(u^n_i) \right], \\
        \Delta^{n,i}_3(\bm{u}^n) &:= \E \left[\int_0^T f(t, X^{n,u^n_i,i}_t, \textsf{W}\mu_t(u^n_i), \alpha^*(t, u^n_i, X^{n,u^n_i,i}_t)) dt + g(X^{n,u^n_i,i}_T, \textsf{W}\mu_T(u^n_i)) \right] \\
        &\hspace{.6cm}- \E \left[\int_0^T f(t, X^{n,u^n_i,i}_t, M^{n,\bm{u}^n,i}_t, \alpha^*(t, u^n_i, X^{n,u^n_i,i}_t)) dt + g(X^{n,u^n_i,i}_T, M^{n,\bm{u}^n,i}_T) \right].
    \end{split} 
\end{equation*}
The first term, $\Delta^{n,i}_1$, is the approximation error incurred when player $i$ substitutes the limiting measure $\textsf{W}\mu_T(u^n_i)$ for the true empirical measure $M^{n,\bm{u}^n,i}_t$, while using the control $\beta$. Similarly for the third term, $\Delta^{n,i}_3$, except now while using the original control $\alpha^*(t,u^n_i,x_i)$.
The second term, $\Delta^{n,i}_2$, compares the control $\beta$ to the control $\alpha^*(t,u^n_i,x_i)$, with the limiting measure in place of the true empirical measure. We will argue that $\Delta^{n,i}_2 \le 0$ thanks to the optimality property of $\alpha^*$, and we will argue that $\Delta^{n,i}_1$ and $\Delta^{n,i}_3$ are small thanks to the convergence of empirical measures.

\begin{lemma} \label{le:Delta2}
We have $\sup_{\beta \in \A_n}\Delta^{n,i}_2(\beta, \bm{u}^n) \le 0$ for a.e.\ $\bm{u}^n \in [0,1]^n$ and all $i \in [n]$.
\end{lemma}
\begin{proof}
Note that $X^{n,u^n_i,i}$ has the same law as $X^{\lambda_{u^n_i},\alpha^*_{u^n_i}}$ as in Lemma \ref{le:optimality_alpha_u}, where $\alpha^*_u(t,x) := \alpha^*(t,u,x)$. Thus  $\Delta^{n,i}_2(\beta,  \bm{u}^n)$ equals
\begin{align*}
\E \left[\int_0^T f(t, X^{n,\beta,u^n_i,i}_t, \textsf{W}\mu_t(u^n_i), \beta_t) dt + g(X^{n,\beta,u^n_i,i}_T, \textsf{W}\mu_T(u^n_i)) \right] - J_W^{u^n_i,\lambda_{u^n_i}}(\mu,\alpha^*_u).
\end{align*}
Recall that $\beta_t=\beta(t, \bm{X}^{n,\beta,\bm{u}^n,i}_t)$ can depend on all $n$ players' state processes, and for this reason the claim is not an immediate consequence of Lemma \ref{le:optimality_alpha_u}. But this issue is resolved by \eqref{pf:DPP-value2}, after noting that the joint law of $(dt\delta_{\beta_t}(da),u^n_i,X^{n,\beta,u^n_i,i})$ belongs to the set $\RC_{u^n_i,\lambda_{u^n_i}}$ defined in the proof of Lemma \ref{le:optimality_alpha_u}. Indeed, we then deduce that 
\begin{align*}
\sup_{\beta \in \A_n}\Delta^{n,i}_2(\beta,  \bm{u}^n) \le \sup_{\beta \in \A_1} J_W^{u^n_i,\lambda_{u^n_i}}(\mu,\beta) - J_W^{u^n_i,\lambda_{u^n_i}}(\mu,\alpha^*_u).
\end{align*}
By Lemma \ref{le:optimality_alpha_u}, this is $\le 0$ for a.e.\ $\bm{u}^n \in [0,1]^n$ and all $i \in [n]$.
\end{proof}

From Lemma \ref{le:optimality_alpha_u}, we deduce that 
\begin{equation*}
\epsilon^n_i(\bm{u}^n) \le \sup_{\beta \in \A_n}\Delta^{n,i}_1(\beta, \bm{u}^n) + \Delta^{n,i}_3(\bm{u}^n), \quad \text{a.e. } \bm{u}^n.
\end{equation*}
Taking averages, we find 
\begin{equation}\label{eq:first_bound_epsilon^n_i}
\frac{1}{n} \sum_{i=1}^n \epsilon^n_i(\bm{u}^n) \le \frac{1}{n}\sum_{i=1}^n \sup_{\beta \in \A_n} \Delta^{n,i}_1(\beta, \bm{u}^n) +  \frac{1}{n} \sum_{i=1}^n \Delta^{n,i}_3(\bm{u}^n).
\end{equation}
Now that we made use of the optimality of $\alpha^*$, it remains to use the convergence results of Section \ref{sec:conv_emp_meas} to show that the right-hand side of \eqref{eq:first_bound_epsilon^n_i} is small.

First, note that $\{\lambda_u : u \in [0,1]\}$ is tight. This is an assumption in Theorems \ref{th:approxEQ-general} and \ref{th:approxEQ-sampling}, and in Theorem \ref{th:approxEQ-continuous} it is a consequence of the assumed continuity of $u \mapsto \lambda_u$.
By boundedness of $b,\sigma$, it is then standard (e.g., using \cite[Theorem 1.4.6]{stroock1997multidimensional}) that the set of laws $\{\L(X^{n,\beta, u^n_i, i}) : n \in \N,  \beta \in \A_n, \bm{u}^n \in [0,1]^n, i \in [n]\}$ is a tight subset of $\P(\C^d)$, where we recall that $\C^d:=C([0,T];\R^d)$. 
Letting $\epsilon > 0$, we may then find a compact set $K \subset \C^d$ such that $\sup_{n,\beta, \bm{u}^n,i} \PP(X^{n, \beta, u^n_i,i} \notin K) \le \epsilon$.
Define the function $h : [0,1] \times \M_+(\C^d) \to \R$ by
\begin{equation}\label{eq:def_h}
\begin{split}
h(u,m) &= \sup_{a \in A} \sup_{z \in K} \left| \int_0^T (f(t,z_t,\textsf{W}\mu_t(u),a) - f(t,z_t, m_t,a))dt \right| \\
    &\hspace{2cm} + \left|g(z_T, \textsf{W}\mu_T(u)) - g(z_T, m_T)\right|,
\end{split}
\end{equation}
where $m_t\in\M_+(\R^d)$ denotes the image of a measure $m \in \M_+(\C^d)$ by the coordinate map $x \mapsto x_t$.
Since $f(t,x,m,a)$ and $g(x,m)$ are bounded, measurable, and continuous in $(x,m,a)$, we deduce that function $h$ bounded and measurable \cite[Theorem 18.19]{aliprantisborder}. Moreover, it follows from compactness of $A$ and $K$ that $h(u,\cdot)$ is continuous on $\M_+(\C^d)$ for each $u \in [0,1]$.
Note that $h(u,\textsf{W}\mu(u))=0$ for every $u$. In order to bound \eqref{eq:first_bound_epsilon^n_i} in terms of $h$, let us choose $C >0$ such that $\max(|f|, |g|) \le C$, and then note that
\begin{equation}\label{eq:inequality_delta_1+3}
\begin{split}
\frac{1}{n}\sum_{i=1}^n \epsilon^n_i(\bm{u}^n) &\le \frac{2}{n} \sum_{i=1}^n \E | h(u^n_i,  M^{n,\bm{u}^n,i})  | + 8 \epsilon C.
\end{split}
\end{equation}
The rest of the argument is different for Theorem \ref{th:approxEQ-general} versus Theorem \ref{th:approxEQ-continuous}.

\subsection{General kernels}

We first prove Theorem \ref{th:approxEQ-general}. Recall that $U^n_i \sim \mathrm{Unif}(I^n_i)$ are independent, and let $\bm{U}^n = (U^n_1, \ldots, U^n_n)$. Abbreviate $\bm{I}^n := I^n_1 \times \cdots \times I^n_n$, and note that $\bm{U}^n$ is uniform on $\bm{I}^n$. Let us also define processes $Y^{n,i}$ such that $(U^n_i,Y^{n,i})_{i \in [n]}$ are independent, with $\L(Y^{n,i}\,|\,U^n_i=u)=\L(X^{\alpha^*}\,|\,U=u)$ for $u \in I^n_i$.
Let $\bm{Y}^n=(Y^{n,1},\ldots,Y^{n,n})$, and define the neighborhood empirical measures (random measures on $\C^d$)
\begin{align*}
\Mni = \frac{1}{n}\sum_{j=1}^n \xi^n_{ij} \delta_{Y^{n,j}}.
\end{align*}
Recall that the process $X^{n,u^n_i,i}$ defined in the beginning of the section is such that $\L(X^{n, u^n_i, i}) = \mu_{u^n_i}$. Recalling that $(U,X^{\alpha^*})$ denotes the equilibrium pair, we have
\begin{align}
\L(X^{\alpha^*}\,|\,U=u^n_i)=\L(Y^{n,i}\,|\,U^n_i=u^n_i)=\L(X^{n, u^n_i, i}) \label{pf:approxEQ-condlaw1}
\end{align}
Hence, for a.e.\ $\bm{u}^n$ and any bounded measurable function $\varphi : \bm{I}^n \times (\C^d)^n \to \R$, we can write
\begin{align*}
\E[\varphi(\bm{u}^n,\bm{X}^{n, \bm{u}^n})] = \E[\varphi(\bm{U}^n,\bm{Y}^n)\,|\,\bm{U}^n=\bm{u}^n], \quad \text{a.e. } \bm{u}^n \in \bm{I}^n.
\end{align*}
In particular, since the empirical measure $M^{n,\bm{u}^n,i}$ defined \eqref{eq:M^n,u^n,i_t} is a functional of $\bm{X}^{n, \bm{u}^n}$, we deduce similarly that 
\begin{align*}
\E[\varphi(\bm{u}^n,M^{n,\bm{u}^n,i})] = \E[\varphi(\bm{U}^n,\Mni)\,|\,\bm{U}^n=\bm{u}^n], \quad \text{a.e. } \bm{u}^n \in \bm{I}^n,
\end{align*}
for bounded measurable $\varphi : \bm{I}^n \times \M_+(\C^d) \to \R$. Applying this in \eqref{eq:inequality_delta_1+3}, along with the tower property, to deduce
\begin{align}
\frac{1}{n}\sum_{i=1}^n \E[\epsilon^n_i(\bm{U}^n)] &\le \frac{2}{n} \sum_{i=1}^n \E | h(U^n_i,  \Mni)  | + 8 \epsilon C. \label{pf:generalkernel1}
\end{align}
The identities \eqref{pf:approxEQ-condlaw1} put us in the setting of Theorem \ref{th:generalprinciple}. As noted above, $h$ is bounded and continuous in its second variable. Hence, Theorem \ref{th:generalprinciple} implies that 
\begin{align*}
\frac{1}{n} \sum_{i=1}^n \E | h(U^n_i,  \Mni)  | \to \E|h(U,\textsf{W}P^{\alpha^*}(U))| = 0,
\end{align*}
where $P^{\alpha^*} := \L(U,X^{\alpha^*})$, with the last identity using the fact that $h(u,\textsf{W}P^{\alpha^*}(u))=0$ for all $u$, which is a consequence of the identity of time-$t$ marginals $(\textsf{W}P^{\alpha^*}(u))_t=\textsf{W}\mu_t$.
Apply this in \eqref{pf:generalkernel1} and then sending $\epsilon \to 0$ completes the proof of Theorem \ref{th:approxEQ-general}. \hfill \qedsymbol

\subsection{Continuous kernels} \label{se:pf:cont-kernel}
We next prove Theorem \ref{th:approxEQ-continuous}. The fact that (1) and (2a--d) imply (2) is a consequence of Corollary \ref{co:cont-disint}. 
The function $h(u,m)$ from \eqref{eq:def_h} is continuous in $m$, uniformly in $u$, because
\begin{align*}
\sup_{u \in [0,1]} |h(u,m') - h(u,m)| &\le \sup_{a \in A} \sup_{z \in K} \left| \int_0^T (f(t,z_t,m'_t,a) - f(t,z_t, m_t,a))dt \right| \\
    &\hspace{2cm} + \left|g(z_T, m'_T) - g(z_T, m_T)\right|,
\end{align*}
and the right-hand side vanishes as $m' \to m$ by compactness of $A$ and $K$ and by joint continuity of $f$ and $g$.
Using also the continuity assumptions of Theorem \ref{th:approxEQ-continuous}, we are therefore in the setting of Theorem \ref{th:generalprinciple-cont}.

Recalling again that $h(u,\textsf{W}P^{\alpha^*}(u))=0$ for all $u$ where again $P^{\alpha^*} := \L(U,X^{\alpha^*})$, Theorem \ref{th:generalprinciple-cont} yields
\begin{align*}
\lim_{n\to\infty} \sup_{\bm{u}^n \in \bm{I}^n} \frac{1}{n} \sum_{i=1}^n \E | h(u^n_i,  M^{n,\bm{u}^n,i})  | = 0.
\end{align*}
Apply this in \eqref{eq:inequality_delta_1+3} and then send $\epsilon \to 0$ to deduce Theorem \ref{th:approxEQ-continuous}. \hfill \qedsymbol

\subsection{Sampling kernels}
We finally prove Theorem \ref{th:approxEQ-sampling}. Again let $P^{\alpha^*} := \L(U,X^{\alpha^*})$, and write $P^{\alpha^*}(du,dx)=duP^{\alpha^*}_u(dx)$ for its disintegration. To prepare for an application of Theorem \ref{th:approxEQ-sampling}, let us first argue that $\{P^{\alpha^*}_u : u \in[0,1]\}$ is tight. Note that $P^{\alpha^*}_u$ is the law of the solution of the SDE
\begin{align*}
dX_t = b(t,X_t,\alpha^*(t,u,X_t))dt + \sigma(t,X_t)dB_t, \quad X_0 \sim \lambda_u.
\end{align*}
Since $b$ and $\sigma$ are bounded and $\{\lambda_u : u \in [0,1]\}$ is tight by assumption, the tightness of $\{P^{\alpha^*}_u : u \in[0,1]\}$ follows easily, e.g., using \cite[Theorem 1.4.6]{stroock1997multidimensional}.

Now, recall that $(u_i)_{i \in \N} \in [0,1]^{\infty}$, where $[0,1]^{\infty}$ is equipped with $($Unif$[0,1])^{\infty}$, and $\xi^n_{ij}=W(u_i,u_j)1_{i \neq j}$ for $i,j \in [n]$ in Theorem \ref{th:approxEQ-sampling}. As in \eqref{eq:inequality_delta_1+3}, we have
\begin{align}
\epsilon^n_i(u_1, \ldots, u_n) &\le 2\E | h(u_i,  N^{n,(u_1, \ldots, u_n)}_i)  | + 8 \epsilon C, \label{pf:apprx-sampling1}
\end{align}
where we define $N^{n,(u_1, \ldots, u_n)}_i = \frac{1}{n}\sum_{j=1,\,j \neq i}^n W(u_i,u_j) \delta_{X^{n,u_j,j}}$. 
Recalling that $h(u,\textsf{W}P^{\alpha^*}(u))=0$ for all $u$, we may thus apply Theorem \ref{th:generalprinciple-sampling} to get
\begin{align*}
\lim_{n\to\infty} \max_{i \in [n]} \E | h(u_i,  N^{n,(u_1, \ldots, u_n)}_i)  | = 0, \quad \text{for a.e. } (u_i)_{i \in \N} \in [0,1]^\infty.
\end{align*}
Combine this with \eqref{pf:apprx-sampling1} and then send $\epsilon \to 0$ to complete the proof. \hfill\qedsymbol

\begin{remark}
Theorem \ref{th:approxEQ-sampling} could likely be strengthened to include a rate of convergence, if one imposed further continuity assumptions on $f$ and $g$. The estimates stemming from Hoeffding's inequality in the proof of Theorem \ref{th:approxEQ-sampling} could, in principle, be traced through to yield exponential bounds on the measure of the set of $(u_1,\ldots,u_n) \in [0,1]^n$ such that $\max_{i \in [n]}\epsilon^n_i(u_1,\ldots,u_n) > \epsilon$. See \cite[Proposition 3]{aurell2021stochastic} for a related result based on a clever application of the law of the iterated logarithm.
\end{remark}

\section{A linear-quadratic example} \label{se:LQ}

In this section we study a linear-quadratic model of \emph{flocking behavior}, inspired by \cite{carmona2013mean,lacker2021case}, which is simple and yet rich enough to exhibit an interesting dependence on the structure of the interaction matrix. This will illustrate also the relative simplicity of our formulation of graphon equilibrium. It should be noted that the model in this section does not fit into the standing assumptions imposed for the theoretical developments in Section \ref{se:mainresults}. However, the definitions of equilibrium require little adaptation for the setting considered below.

We work in dimension $d=1$. We shall now assume $W \in L^2_+[0,1]^2$, i.e., the kernel is square-integrable.
For ${m} \in \P([0,1] \times \R)$, recall the definition of the measure-valued function $\textsf{W}{m} : [0,1] \to \M_+(\R)$ from \eqref{def:Woperator-measure}. We define its mean $\overline{\textsf{W}}{m} : [0,1] \to \R$ by
\begin{align*}
\overline{\textsf{W}}{m}(u) = \int_{[0,1] \times \R} W(u,v) \, x \, {m}(dv,dx),
\end{align*}
whenever this integral is well-defined.
The linear-quadratic model we study can be summarized concisely, as in \eqref{def:graphonEQ-compact}, as follows:
\begin{equation}
    \left\{ \begin{array}{ll}
        \inf_{\alpha \in \mathcal{A}} & \frac{1}{2} \E \left[\int_0^T \alpha_t^2 dt + c|X_T- \overline{\textsf{W}}\mu_T(U)|^2 \right] \\
        \mbox{s.t.} & dX_t = \alpha_t dt + \sigma dB_t, \\
        & \mu_t = \L(U, X_t),  \ \ (U, X_0) \sim \lambda \in \Punif([0,1] \times \R).
    \end{array}
    \right. \label{def:LQ-compact}
\end{equation}
Note that, in equilibrium, $\overline{\textsf{W}}\mu_T(u)=\E[W(u,U)X_T]$ for a.e.\ $u$.
In the notation of Section \ref{se:mainresults}, we are choosing $A=\R$ and
\begin{align*}
b(t,x,a)&=a, \quad \sigma(t,x)=\sigma, \quad f(t,x,m,a) = -\frac12 a^2, \quad g(x,m) = -c\Big(x - \int_{\R} x\,m(dx)\Big)^2.
\end{align*}

\begin{proposition}
Assume $W \in L^2_+[0,1]^2$ satisfies $\|W\|_{L^2[0,1]^2} < 1+(cT)^{-1}$. Assume $\lambda$ has a finite second moment:
\begin{align*}
\int_{[0,1] \times \R} x^2\, \lambda(du,dx) < \infty.
\end{align*}
Then there exists a $W$-equilibrium with associated control given by
\begin{align*}
\alpha(t,u,x) &= \frac{c}{c(T-t)+1}(M(u) - x),
\end{align*}
where $M \in L^2[0,1]$ is defined by
\begin{align}
M := \tfrac{1}{cT+1}\text{\rm\textsf{W}}(\text{\rm\textsf{Id}}-\tfrac{cT}{cT+1}\text{\rm\textsf{W}})^{-1}\psi , \label{def:LQ-M}
\end{align}
with $\psi \in L^2[0,1]$ defined by $\psi(u):=\E[X_0\,|\,U=u]$. Here, $\text{\rm\textsf{Id}}$ is the identity operator and $\text{\rm\textsf{W}}$ is viewed as an operator on $L^2[0,1]$ as defined in \eqref{def:Woperator-function}.
\end{proposition}

The assumption that $\|W\|_{L^2[0,1]^2} < 1+(cT)^{-1}$ ensures the existence of the inverse operator appearing in \eqref{def:LQ-M}. Equilibria may fail to exist without this assumption. Indeed, if $W \equiv 1+cT$, then the proof shows that there is no solution, unless $\E[X_0\,|\,U]=0$ a.s., in which case the solution is as above with $M \equiv 0$.

There is a notable appearance here of a common notion of \emph{centrality} used in graph theory. 

If $X_0$ and $U$ are independent, then $\psi \equiv \E[X_0]$ and so
\begin{align*}
M &= \E[X_0]\tfrac{1}{cT+1} \textsf{W} (\textsf{Id}-\tfrac{cT}{cT+1} \textsf{W} )^{-1}\bm{1} \\
	&= \tfrac{1}{cT}\E[X_0] \big[(\textsf{Id}-\tfrac{cT}{cT+1} \textsf{W} )^{-1} - \textsf{Id}\big]\bm{1},
\end{align*}
where $\bm{1}$ is the constant function equal to 1. 
The quantity $\big[(\textsf{Id}-\tfrac{cT}{cT+1} \textsf{W} )^{-1} - \textsf{Id}\big]\bm{1}(u)$ is precisely the \emph{Katz centrality} or \emph{$\alpha$-centrality} of the vertex $u \in [0,1]$, or rather the infinite-dimensional (graphon) analogue thereof, with parameter $\alpha=cT/(cT+1)$.
When $X_0$ and $U$ are not independent, we have instead a generalization of this centrality concept in which a vertex $u$ receives a weight proportional to the mean initial position $\psi(u)$.
Note if $X_0=h(U)$ is $U$-measurable, then $\psi = h$.

\subsection{Derivation of the solution}
We follow roughly the PDE approach discussed in Section \ref{se:PDE}.
We fix for now a mean field term and compute the best response. That is, we fix for now a measurable function $M : [0,1] \to \R$, to play the role of the mean function $\overline{\textsf{W}}\mu_T$. The stochastic control problem in \eqref{def:LQ-compact} is associated with the HJB equation
\begin{align*}
\partial_t v(t,u,x) - \frac12|\partial_x v(t,u,x)|^2 + \frac{\sigma^2}{2}\partial_{xx}v(t,u,x) = 0, \quad v(T,u,x) = c\big(x - M(u)\big)^2.
\end{align*}
The corresponding optimal control is $\alpha(t,u,x) = -\partial_xv(t,u,x)$.
We solve this PDE explicitly using the ansatz $v(t,u,x) = \psi(t)+\frac12\varphi(t)(x-M(u))^2$, where $\varphi$ and $\psi$ are functions on $[0,T]$ to be determined. Plugging this ansatz into the HJB, we obtain that $\varphi$ and $\psi$ should satisfy
\begin{equation*}
    \frac{1}{2}(x- M(u))^2(\varphi'(t) - \varphi(t)^2) + \frac{\sigma^2}{2} \varphi(t) + \psi'(t) = 0,
\end{equation*}
for all $(t,u,x) \in (0,T) \times [0,1] \times \R$, along with the terminal conditions $\varphi(T)=c$ and $\psi(T)=0$. Matching coefficients, we find
\begin{align*}
\varphi'(t) &= \varphi^2(t), \qquad 
\psi'(t) = - \frac{\sigma^2}{2} \varphi(t).
\end{align*}
This system is easily solved using the aforementioned boundary conditions:
\begin{equation*}
\varphi(t) = \frac{c}{c(T-t) + 1}, \qquad \psi(t) = \frac{\sigma^2}{2} \log (c(T-t)+1).
\end{equation*}
The optimal control is thus given by $\alpha^*(t, u, x) = - \partial_xv(t,u,x) = \frac{c}{c(T-t)+1} (M(u)-x)$, and the optimal state process thus satisfies the following dynamics
\begin{equation*}
dX_t = \frac{c}{c(T-t)+1} (M(U) - X_t)dt + \sigma dB_t, \qquad (U,X_0) \sim \lambda.
\end{equation*}
Define $\mu_t = \L(U,X_t)$ for each $t \in [0,T]$. Then $\mu$ is a graphon equilibrium if and only if $M(u) = \overline{\textsf{W}}\mu_T(u)$, i.e., $M(u) = \E[W(u,U)X_T]$, for a.e.\ $u \in [0,1]$.  In other words, we will have an equilibrium if we can solve the (McKean-Vlasov) SDE
\begin{equation}\label{eq:dynamics_X_eq_LQ}
\begin{split}
dX_t &= \frac{c}{c(T-t)+1} (\overline{\textsf{W}}\mu_T(U) - X_t)dt + \sigma dB_t, \qquad (U,X_0) \sim \lambda, \\
\mu_t &= \L(U,X_t), \quad t \in [0,T].
\end{split}
\end{equation}
To solve this equation, it is convenient to introduce an independent copy $(\widetilde{B},\widetilde{U},\widetilde{X})$ of $(B,U,X)$.
As a first step, we find an expression for $\overline{\textsf{V}}\mu_T(U)$ for every kernel $V \in L^2_+[0,1]^2$, where we note by definition that
\begin{equation}
\overline{\textsf{V}}\mu_T(U) = \int_{[0,1] \times \R} V(U,v) \, x \, \mu_T(dv,dx) = \E[V(U,\widetilde{U})\widetilde{X}_T\,|\,U].  \label{eq:FP1}
\end{equation}
To find an expression for this, note that the SDE \eqref{eq:dynamics_X_eq_LQ} implies
\begin{align*}
\widetilde{X}_t = \widetilde{X}_0 + \int_0^t\frac{c}{c(T-s)+1}\big( \overline{\textsf{W}}\mu_T(\widetilde{U})  -  \widetilde{X}_s\big)\,ds + \sigma \widetilde{B}_t.
\end{align*}
Multiply by $V(U,\widetilde{U})$ and take conditional expectations given $U$, using independence of $\widetilde{B}$ and $\widetilde{U}$, to get
\begin{align} \begin{split}
\overline{\textsf{V}}\mu_t(U) &= \E[V(U,\widetilde{U})\widetilde{X}_t\,|\,U] \\
		&= \overline{\textsf{V}}\mu_0(U)  + \int_0^t\frac{c}{c(T-s)+1}\big(\E[V(U,\widetilde{U})\overline{\textsf{W}}\mu_T(\widetilde{U})\,|\,U]  - \overline{\textsf{V}}\mu_s(U)\big)\,ds. \end{split} \label{pf:Vmu1}
\end{align}
The second to last term simplifies by Fubini's theorem:
\begin{align*}
\E[V(U,\widetilde{U})\overline{\textsf{W}}\mu_T(\widetilde{U})\,|\,U] &= \int_0^1V(U,\widetilde{u})\overline{\textsf{W}}\mu_T(\widetilde{u})\,d\widetilde{u} \\
	&= \int_0^1V(U,\widetilde{u})\int_{[0,1]\times \R} W(\widetilde{u},v)\,x\,\mu_T(dv,dx)\,d\widetilde{u} \\
	&= \int_{[0,1]\times \R} V\circ W(U,v)\,x\,\mu_T(dv,dx).
\end{align*}
Here we define $V \circ W \in L^2_+[0,1]^2$ by $V\circ W(u,v) := \int_0^1 V(u,\widetilde{u})W(\widetilde{u} ,v) \,d\widetilde{u}$, which is exactly the kernel of the composition operator $\textsf{V} \circ \textsf{W}$, which we abbreviate as $\textsf{VW}$. We may thus write
\begin{align*}
\E[V(U,\widetilde{U})\overline{\textsf{W}}\mu_T(\widetilde{U})\,|\,U] = \overline{\textsf{VW}}\mu_T(U).
\end{align*}
Use this identity and differentiate \eqref{pf:Vmu1} to find that $(\overline{\textsf{V}}\mu_t(u))_{t \in [0,T]}$ obeys the differential equation
\begin{align*}
\frac{d}{dt}\overline{\textsf{V}}\mu_t(u) = \frac{c}{c(T-t)+1}\big(\overline{\textsf{VW}}\mu_T(u) - \overline{\textsf{V}}\mu_t(u)\big).
\end{align*}

It follows that $\overline{\textsf{V}}\mu_t(u)$ must be of the form
\begin{align*}
\overline{\textsf{V}}\mu_t(u) = \overline{\textsf{VW}}\mu_T(u) + \kappa(u)(c(T-t)+1), \quad t \in [0,T],
\end{align*}
for a $u$-dependent parameter $\kappa(u)$ to be determined by the initial conditions. Setting $t=0$ implies $\kappa(u)=\frac{1}{cT+1}(\overline{\textsf{V}}\mu_0(u) - \overline{\textsf{VW}}\mu_T(u))$, and thus
\begin{align}
\overline{\textsf{V}}\mu_t(u) = \frac{ct}{cT+1}\overline{\textsf{VW}}\mu_t(u)  + \frac{c(T-t)+1}{cT+1}\overline{\textsf{V}}\mu_0(u). \label{pf:VW2}
\end{align}
In particular, setting $t=T$, and noting that $\overline{\textsf{V}}\mu_T(u)$ depends linearly on the operator $\textsf{V}$, we find
\begin{align}
\overline{\textsf{V}\big(\textsf{Id}-\tfrac{cT}{cT+1}\textsf{W}\big)}\mu_T(u) = \tfrac{1}{cT+1}\overline{\textsf{V}}\mu_0(u). \label{pf:VW3}
\end{align}
Choosing $\textsf{V}=\textsf{W}(\textsf{Id}-\tfrac{cT}{cT+1}\textsf{W})^{-1}$ yields
\begin{align}
\overline{\textsf{W}}\mu_T(u) = \tfrac{1}{cT+1}\overline{\textsf{W}(\textsf{Id}-\tfrac{cT}{cT+1}\textsf{W})^{-1}}\mu_0(u). \label{pf:VW4}
\end{align}

Note also that $\mu_0=\L(U,X_0)$, and so for any kernel $V$ we have
\begin{align*}
\overline{\textsf{V}}\mu_0(u) = \E[ V(u,U)X_0 ] = \E[V(u,U)\psi(U)] = \textsf{V}\psi(u),
\end{align*}
where $\psi(u):=\E[X_0\,|\,U=u]$. Combining this with \eqref{pf:VW4} shows that $M(u)=\overline{\textsf{W}}\mu_T(u)$ is given by \eqref{def:LQ-M}. \hfill \qedsymbol

\bibliographystyle{amsplain}
\bibliography{biblio_network}

\providecommand{\bysame}{\leavevmode\hbox to3em{\hrulefill}\thinspace}
\providecommand{\MR}{\relax\ifhmode\unskip\space\fi MR }
\providecommand{\MRhref}[2]{%
  \href{http://www.ams.org/mathscinet-getitem?mr=#1}{#2}
}
\providecommand{\href}[2]{#2}
\begin{thebibliography}{10}

\bibitem{aliprantisborder}
C.~Aliprantis and K.~Border, \emph{Infinite dimensional analysis: {A}
  hitchhiker's guide}, 3 ed., Springer, 2007.

\bibitem{aurell2021stochastic}
A.~Aurell, R.~Carmona, and M.~Lauriere, \emph{Stochastic graphon games: {II}.
  {T}he linear-quadratic case}, arXiv preprint arXiv:2105.12320 (2021).

\bibitem{basak2017universality}
A.~Basak and S.~Mukherjee, \emph{Universality of the mean-field for the {P}otts
  model}, Probability Theory and Related Fields \textbf{168} (2017), no.~3-4,
  557--600.

\bibitem{bayraktar2020graphon}
E.~Bayraktar, S.~Chakraborty, and R.~Wu, \emph{Graphon mean field systems},
  arXiv preprint arXiv:2003.13180 (2020).

\bibitem{bayraktar2022propagation}
E.~Bayraktar, R.~Wu, and X.~Zhang, \emph{Propagation of chaos of
  forward-backward stochastic differential equations with graphon
  interactions}, arXiv preprint arXiv:2202.08163 (2022).

\bibitem{beiglbock2018denseness}
M.~Beiglb{\"o}ck and D.~Lacker, \emph{Denseness of adapted processes among
  causal couplings}, arXiv preprint arXiv:1805.03185 (2018).

\bibitem{bertsekas1996stochastic}
D.P. Bertsekas and S.E. Shreve, \emph{Stochastic optimal control: the
  discrete-time case}, vol.~5, Athena Scientific, 1996.

\bibitem{bet2020weakly}
G.~Bet, F.~Coppini, and F.R. Nardi, \emph{Weakly interacting oscillators on
  dense random graphs}, arXiv preprint arXiv:2006.07670 (2020).

\bibitem{bhamidi2019weakly}
S.~Bhamidi, A.~Budhiraja, and R.~Wu, \emph{Weakly interacting particle systems
  on inhomogeneous random graphs}, Stochastic Processes and their Applications
  \textbf{129} (2019), no.~6, 2174--2206.

\bibitem{bogachev2007measure}
V.I. Bogachev, \emph{Measure theory}, Springer, 2007.

\bibitem{borgs-chayes-cohn-zhao-I}
C.~Borgs, J.~Chayes, H.~Cohn, and Y.~Zhao, \emph{An $l^p$ theory of sparse
  graph convergence {I}: {L}imits, sparse random graph models, and power law
  distributions}, Transactions of the American Mathematical Society
  \textbf{372} (2019), no.~5, 3019--3062.

\bibitem{brunick2013mimicking}
G.~Brunick and S.~Shreve, \emph{Mimicking an {I}t{\^o} process by a solution of
  a stochastic differential equation}, The Annals of Applied Probability
  \textbf{23} (2013), no.~4, 1584--1628.

\bibitem{caines2019graphon}
P.-E. Caines and M.~Huang, \emph{Graphon mean field games and the {GMFG}
  equations: $\varepsilon$-{N}ash equilibria}, 2019 IEEE 58th Conference on
  Decision and Control (CDC), IEEE, 2019, pp.~286--292.

\bibitem{caines2018graphon}
P.E. Caines and M.~Huang, \emph{Graphon mean field games and the {GMFG}
  equations}, 2018 IEEE Conference on Decision and Control (CDC), IEEE, 2018,
  pp.~4129--4134.

\bibitem{cardaliaguet2019master}
P.~Cardaliaguet, F.~Delarue, J.-M. Lasry, and P.-L. Lions, \emph{The master
  equation and the convergence problem in mean field games}, Princeton
  University Press, 2019.

\bibitem{carmona2004nash}
G.~Carmona, \emph{Nash equilibria of games with a continuum of players},
  (2004).

\bibitem{carmona2019stochastic}
R.~Carmona, D.~Cooney, C.~Graves, and M.~Lauriere, \emph{Stochastic graphon
  games: {I}. {T}he static case}, arXiv preprint arXiv:1911.10664 (2019).

\bibitem{carmona-delarue}
R.~Carmona and F.~Delarue, \emph{Probabilistic analysis of mean-field games},
  SIAM Journal on Control and Optimization \textbf{51} (2013), no.~4,
  2705--2734.

\bibitem{carmona-delarue-book}
\bysame, \emph{Probabilistic theory of mean field games with applications
  {I}-{II}}, Springer, 2018.

\bibitem{carmona2013mean}
R.~Carmona, J.-P. Fouque, and L.-H. Sun, \emph{Mean field games and systemic
  risk}, Available at SSRN 2307814 (2013).

\bibitem{coppini2019long}
F.~Coppini, \emph{Long time dynamics for interacting oscillators on graphs},
  arXiv preprint arXiv:1908.01520 (2019).

\bibitem{coppini2021note}
\bysame, \emph{A note on {F}okker-{P}lanck equations and graphons}, arXiv
  preprint arXiv:2102.04505 (2021).

\bibitem{coppini2019law}
F.~Coppini, H.~Dietert, and G.~Giacomin, \emph{A law of large numbers and large
  deviations for interacting diffusions on erd{\"{o}}s--r{\'e}nyi graphs},
  Stochastics and Dynamics (2019), 2050010.

\bibitem{cui2021learning}
K.~Cui and H.~Koeppl, \emph{Learning graphon mean field games and approximate
  {a}sh equilibria}, arXiv preprint arXiv:2112.01280 (2021).

\bibitem{delarue2017mean}
F.~Delarue, \emph{Mean field games: {A} toy model on an {E}rd{\"o}s-{R}enyi
  graph.}, ESAIM: Proceedings and Surveys \textbf{60} (2017), 1--26.

\bibitem{delattre2016note}
S.~Delattre, G.~Giacomin, and E>~Lu{\c{c}}on, \emph{A note on dynamical models
  on random graphs and fokker--planck equations}, Journal of Statistical
  Physics \textbf{165} (2016), no.~4, 785--798.

\bibitem{nicole1987compactification}
N.~{El Karoui}, D.H. Nguyen, and M.~Jeanblanc-Picqu{\'e},
  \emph{Compactification methods in the control of degenerate diffusions:
  existence of an optimal control}, Stochastics: an international journal of
  probability and stochastic processes \textbf{20} (1987), no.~3, 169--219.

\bibitem{fan1952fixed}
K.~Fan, \emph{Fixed-point and minimax theorems in locally convex topological
  linear spaces}, Proceedings of the National Academy of Sciences of the United
  States of America \textbf{38} (1952), no.~2, 121.

\bibitem{fouque-recent}
Y.~Feng, J.-P. Fouque, and T.~Ichiba, \emph{Linear-quadratic stochastic
  differential games on directed chain networks}, arXiv preprint
  arXiv:2003.08840 (2020).

\bibitem{gao2020linear}
S.~Gao, R.~{Foguen Tchuendom}, and P.E. Caines, \emph{Linear quadratic graphon
  field games}, arXiv preprint arXiv:2006.03964 (2020).

\bibitem{haussmann1990existence}
U.G. Haussmann and J.P. Lepeltier, \emph{On the existence of optimal controls},
  SIAM Journal on Control and Optimization \textbf{28} (1990), no.~4, 851--902.

\bibitem{huang2006large}
M.~Huang, R.P. Malham{\'e}, and P.E. Caines, \emph{Large population stochastic
  dynamic games: closed-loop {M}c{K}ean-{V}lasov systems and the {N}ash
  certainty equivalence principle}, Communications in Information \& Systems
  \textbf{6} (2006), no.~3, 221--252.

\bibitem{jabin2021mean}
P.-E. Jabin, D.~Poyato, and J.~Soler, \emph{Mean-field limit of
  non-exchangeable systems}, arXiv preprint arXiv:2112.15406 (2021).

\bibitem{jackson2010social}
M.O. Jackson, \emph{Social and economic networks}, Princeton university press,
  2010.

\bibitem{lacker2015mean}
D.~Lacker, \emph{Mean field games via controlled martingale problems:
  {E}xistence of {M}arkovian equilibria}, Stochastic Processes and their
  Applications \textbf{125} (2015), no.~7, 2856--2894.

\bibitem{lacker2018mean}
\bysame, \emph{Mean field games and interacting particle systems}, Preprint
  (2018).

\bibitem{lacker2020convergence}
\bysame, \emph{On the convergence of closed-loop {N}ash equilibria to the mean
  field game limit}, The Annals of Applied Probability \textbf{30} (2020),
  no.~4, 1693--1761.

\bibitem{lacker2021case}
D.~Lacker and A.~Soret, \emph{A case study on stochastic games on large graphs
  in mean field and sparse regimes}, Mathematics of Operations Research (2021).

\bibitem{lasry-lions}
J.-M. Lasry and P.-L. Lions, \emph{Mean field games}, Japanese Journal of
  Mathematics \textbf{2} (2007), no.~1, 229--260.

\bibitem{lovasz2012large}
L.~Lov{\'a}sz, \emph{Large networks and graph limits}, vol.~60, American
  Mathematical Soc., 2012.

\bibitem{luccon2020quenched}
E.~Lu{\c{c}}on, \emph{Quenched asymptotics for interacting diffusions on
  inhomogeneous random graphs}, Stochastic Processes and their Applications
  \textbf{130} (2020), no.~11, 6783--6842.

\bibitem{parise2019graphon}
F.~Parise and A.~Ozdaglar, \emph{Graphon games}, Proceedings of the 2019 ACM
  Conference on Economics and Computation, 2019, pp.~457--458.

\bibitem{parise2021analysis}
\bysame, \emph{Analysis and interventions in large network games}, Annual
  Review of Control, Robotics, and Autonomous Systems \textbf{4} (2021),
  455--486.

\bibitem{stroock1997multidimensional}
D.W. Stroock and S.R.S. Varadhan, \emph{Multidimensional diffusion processes},
  vol. 233, Springer Science \& Business Media, 1997.

\bibitem{sun2006exact}
Y.~Sun, \emph{The exact law of large numbers via {F}ubini extension and
  characterization of insurable risks}, Journal of Economic Theory \textbf{126}
  (2006), no.~1, 31--69.

\bibitem{sznitman1991topics}
A.-S. Sznitman, \emph{Topics in propagation of chaos}, Ecole d'{\'e}t{\'e} de
  probabilit{\'e}s de Saint-Flour XIX—1989, Springer, 1991, pp.~165--251.

\bibitem{tangpi2022optimal}
L.~Tangpi and X.~Zhou, \emph{Optimal investment in a large population of
  competitive and heterogeneous agents}, arXiv preprint arXiv:2202.11314
  (2022).

\bibitem{vasal2020sequential}
D.~Vasal, R.K. Mishra, and S.~Vishwanath, \emph{Sequential decomposition of
  graphon mean field games}, arXiv preprint arXiv:2001.05633 (2020).

\bibitem{veretennikov1981strong}
A.J. Veretennikov, \emph{On strong solutions and explicit formulas for
  solutions of stochastic integral equations}, Sbornik: Mathematics \textbf{39}
  (1981), 387--403.

\end{thebibliography}

\end{document}